\documentclass[12pt]{amsart}
\usepackage{amssymb,amscd}
\usepackage{verbatim}

\usepackage{amsmath,graphicx,mathrsfs} % my new
\usepackage[colorlinks=true,allcolors = blue]{hyperref} % my new

\let\frak\mathfrak

\def\>{\relax\ifmmode\mskip.666667\thinmuskip\relax\else\kern.111111em\fi}
\def\<{\relax\ifmmode\mskip-.333333\thinmuskip\relax\else\kern-.0555556em\fi}
\def\vsk#1>{\vskip#1\baselineskip}
\def\vv#1>{\vadjust{\vsk#1>}\ignorespaces}
\def\vvn#1>{\vadjust{\nobreak\vsk#1>\nobreak}\ignorespaces}

\let\dsize\displaystyle \let\tsize\textstyle

\newtheorem{thm}{Theorem}[section]
\newtheorem{cor}[thm]{Corollary}
\newtheorem{lem}[thm]{Lemma}

\theoremstyle{definition} % My June 14 2017
\newtheorem{exmp}{Example}[section]

\numberwithin{equation}{section}

\theoremstyle{definition}
\newtheorem*{rem}{Remark}

\let\mc\mathcal
\let\nc\newcommand

\let\al\alpha

\let\eps\varepsilon

\let\ka\kappa
\let\la\lambda
\let\La\Lambda

\let\phi\varphi
\let\si\sigma

\let\ox\otimes

\let\ge\geqslant
\let\geq\geqslant
\let\le\leqslant
\let\leq\leqslant

\let\on\operatorname
\let\bi\bibitem
\let\bs\boldsymbol

\def\C{{\mathbb C}}
\def\Z{{\mathbb Z}}

\def\F{{\mathbb F}}

\def\End{\on{End}}

\def\gl{\mathfrak{gl}}

\def\beq{\begin{equation}}
\def\eeq{\end{equation}}
\def\be{\begin{equation*}}
\def\ee{\end{equation*}}

\nc{\bea}{\begin{eqnarray*}}
\nc{\eea}{\end{eqnarray*}}
\nc{\bean}{\begin{eqnarray}}
\nc{\eean}{\end{eqnarray}}
\let\Ga\Gamma

\nc{\Il}{{\mc I_{\bs\la}}}
\nc{\bla}{{\bs\la}}
\nc{\Fla}{\F_\bla}
\nc{\tfl}{{T^*\Fla}}
\nc{\GL}{{GL_n(\C)}}
\nc{\GLC}{{GL_n(\C)\times\C^*}}

\def\KZ/{{\slshape KZ\/}}
\def\qKZ/{{\slshape qKZ\/}}
\def\XXX/{{\slshape XXX\/}}

\def\Sing{{\on{Sing}}}

\def\K{{\mathbb K}}
\def\A{{\mathbb A}}

\begin{document}

\hrule width0pt
\vsk->

\title[Eigenvectors of $p$-Curvature for Geometric Difference Equations]
{Eigenvectors of $p$-Curvature Operators for
\\
Geometric Difference Equations}

\author[Vitaly Tarasov and Alexander Varchenko]
{Vitaly Tarasov$\>^\circ$ and Alexander Varchenko$\>^\star$}

\maketitle

\begin{center}
{\it $^\circ$Department of Mathematical Sciences,
Indiana University,
402 North Blackford St,
\\
Indianapolis, IN 46202-3216, USA\/}
\vsk.5>
{\it $^{\star}\<$Department of Mathematics, University
of North Carolina at Chapel Hill\\ Chapel Hill, NC 27599-3250, USA\/}
\end{center}

{\let\thefootnote\relax
\footnotetext{\vsk-.8>\noindent
$^\circ\<${\sl E\>-mail}:\enspace vtarasov@iu.edu\>,
supported in part by Simons Foundation grants \rlap{430235, 852996}
\\
$^\star\<${\sl E\>-mail}:\enspace anv@email.unc.edu\>,
supported in part by Simons Foundation grant TSM-00012774}}

\vsk>
{\leftskip3pc \rightskip\leftskip \parindent0pt \Small
{\it Key words\/}:
Discrete flat qKZ connection, commuting operators, integral representation,
discrete integral, $p$-curvature
\vsk.6>
{\it 2020 Mathematics Subject Classification\/}: 81R50, 33C70, 33D80
%33C70, 55N91, 33D80 %14N35, 53D45, 14D05, 33C70
\par}

\begin{abstract}

A qKZ-type additive discrete flat connection in characteristic $p$ has $p$-curvature operators. They
are commuting automorphisms of the connection. We present a Bethe-ansatz-type
construction of eigensections and eigenvalues of the $p$-curvature operators.
An eigensection of $p$-curvature operators is a discrete
hypergeometric sum over the finite lattice
$\Z^r/p\Z^r$.
Thus the $p$-curvature eigensections are constructed by a finite, discrete analogue of a
hypergeometric integral.

\end{abstract}

{\small\tableofcontents\par}

\setcounter{footnote}{0}
\renewcommand{\thefootnote}{\arabic{footnote}}

\section{Introduction}
\label{sec 1}

Discrete qKZ-type flat connections play a fundamental role in representation theory, integrable systems, and the geometry of Nakajima varieties. When such a connection is reduced to a field of characteristic $p$, it acquires a new family of commuting automorphisms, the $p$-curvature operators.
The purpose of this paper is to construct common eigenvectors and eigenvalues of these operators
by developing an analogue of the Bethe ansatz
method based on discrete Gauss--Manin integral representations.

In this Bethe ansatz, the eigenvectors are not given simply by the value of the weight
function at a solution of the Bethe ansatz equations, as in the classical Bethe ansatz theory. Instead, they are
given by a discrete hypergeometric sum over the finite lattice
$\Z^r/p\Z^r$.
These sums are finite-characteristic analogues of hypergeometric integral representations:
they combine the values of the weight function along a finite orbit with coefficients
given by discrete parallel transport.

\smallskip
The standard setting over $\C$
is as follows. Let $V$ be a complex vector space, let $a \in \mathbb{C}$, and let $x=(x_1,\dots,x_n)$. Suppose we are given $\on{GL}(V)$-valued functions
\[
K_m(x,a), \qquad m=1,\dots,n,
\]
satisfying the flatness relations
\begin{equation}
\label{flatness}
K_m(x_1,\dots,x_l+a,\dots,x_n,a)\,K_l(x,a)
=
K_l(x_1,\dots,x_m+a,\dots,x_n,a)\,K_m(x,a).
\end{equation}
Then these operators define an additive discrete flat connection with step $a$ and fiber $V$ on the space with coordinates $x$ for every fixed $a$. A flat section $I(x,a)$ is a $V$-valued solution of the system
\begin{equation}
I(x_1,\dots,x_m+a,\dots,x_n,a)=K_m(x,a)\,I(x,a), \qquad m=1,\dots,n.
\label{qqKZ}
\end{equation}
A fundamental problem is to construct integral representations of flat sections, or equivalently, to realize the discrete connection as a discrete Gauss--Manin connection.

\smallskip
When $a=0$, the operators $(K_m(x,0))$ commute, and one naturally seeks their joint eigenvectors and eigenvalues.

\smallskip

Equations \eqref{qqKZ} arise, for example, as rational qKZ equations in representation
theory and as quantum qKZ equations in the equivariant cohomology of Nakajima varieties.
In the case $a=0$, the operators $(K_m(x,0))$ appear as commuting transfer matrices of the XXX
integrable chain.
See, for instance, \cite{S, FR, EFK, AO, MO}.

\smallskip

Integral representations are known for rational qKZ equations
and for qKZ difference equations in
the equivariant cohomology of
Nakajima varieties. Such a representation consists of a scalar, multi-valued
master function $F(s,x,a)$
and a $V$-valued weight function $W(s,x)$, where $s=(s_1,\dots,s_r)$ is a collection of
auxiliary variables. The corresponding integrals
\[
I(x,a)=\int F(s,x,a)\,W(s,x)\,ds
\]
solve \eqref{qqKZ}.

The operators $(K_m(x,0))$ are diagonalized by the Bethe ansatz. One introduces the scalar functions
\[
\Phi_l(s,x,a)=\frac{F(s_1,\dots,s_l+a,\dots,s_r,x,a)}{F(s,x,a)},
\quad
\Psi_m(s,x,a)=\frac{F(s,x_1,\dots,x_m+a,\dots,x_n,a)}{F(s,x,a)}.
\]
These functions satisfy the corresponding flatness relations and define a discrete flat
connection of rank one on the space with coordinates $s$ and $x$ for every fixed $a$.

If $(s^0,x^0)$ solves the system of Bethe ansatz equations
\[
\Phi_l(s,x,0)=1, \qquad l=1,\dots,r,
\]
then $W(s^0,x^0)$ is an eigenvector of $K_m(x^0,0)$ with eigenvalue $\Psi_m(s^0,x^0,0)$:
\[
K_m(x^0,0)\,W(s^0,x^0)=\Psi_m(s^0,x^0,0)\,W(s^0,x^0), \qquad m=1,\dots,n.
\]

\vsk.2>

We now pass to finite characteristic.
In this paper, we consider discrete flat
connections over a field $\K$ of characteristic $p>2$.
Let $V_\K$ be a vector space over $\K$, let $a\in\K^\times$, and let
\bea
K_{m,\K}(x,a), \qquad m=1,\dots, n,
\eea
be
$\on{GL}(V_\K)$-valued functions satisfying the flatness relations \eqref{flatness}.
Such a discrete flat connection $(K_{m,\K}(x,a))$ has
$p$-curvature operators defined by
\bea
\tilde K_{m,\K}(x,a)
&=&
K_{m,\K}(x_1,\dots,x_m +(p-1)a,\dots,x_n, a)\cdots
\\
&&
\cdots
K_{m,\K}(x_1,\dots,x_m+a,\dots,x_n, a)\,K_{m,\K}(x,a).
\eea
These operators commute:
\[
\tilde K_{l,\K}(x,a)\,\tilde K_{m,\K}(x,a)
=
\tilde K_{m,\K}(x,a)\,\tilde K_{l,\K}(x,a),
\qquad 1\le l,m\le n,
\]
and each $\tilde K_{m,\K}(x,a)$ is an automorphism of the discrete connection:
\[
\tilde K_{m,\K}(x_1,\dots,x_l+a,\dots,x_n,a)\,K_{l,\K}(x,a)
=
K_{l,\K}(x,a)\,\tilde K_{m,\K}(x,a).
\]

\smallskip

The main result of this paper is a construction of common eigenvectors and eigenvalues of the
$p$-curvature operators $(\tilde K_{m,\K}(x,a))$
using an integral representation of the discrete flat connection $(K_{m,\K}(x,a))$.

\vsk.2>

Namely, we introduce the notion of integral representation for $(K_{m,\K}(x,a))$
that formalizes, in finite characteristic, the notion of integral representation for a discrete flat connection
$(K_m(x,a))$ over $\C$.
The main objects are the $V_\K$-valued weight
function $W_\K(s,x)$ and the scalar functions
\[
\Phi_{l,\K}(s,x,a), \qquad \Psi_{m,\K}(s,x,a).
\]
They are
defined by the same formulas as in the complex case, reduced to characteristic
$p$.

\vsk.2>

We then define the scalar $p$-curvature functions of the rank-one discrete flat connection
$(\Phi_{l,\K}(s,x,a),\Psi_{m,\K}(s,x,a))$ by
\begin{align*}
\tilde{\Phi}_{l,\K}(s,x,a) &= \prod_{i=0}^{p-1} \Phi_{l,\K}(s_1,\dots, s_l+ia,\dots, s_r, x, a),
\\
\tilde{\Psi}_{m,\K}(s,x,a) &= \prod_{i=0}^{p-1} \Psi_{m,\K}(s, x_1,\dots, x_m + ia, \dots, x_n,a).
\end{align*}
Given $a\in \K^\times$, we show that if $(s^0,x^0)$ satisfies the system
\bean
\label{gbae}
\tilde \Phi_{l,\K}(s,x,a)=1,
\qquad l=1,\dots,r,
\eean
then the vector
\bean
\label{mx00}
J(s^0,x^0,a)
&=&
\sum_{j_1,\dots,j_r=0}^{p-1}
W_\K(s^0_1+j_1a, \dots, s^0_r+j_ra,x^0)\,
\\
\notag
&\times &
\prod_{i=1}^r\prod_{d_i=0}^{j_i-1}
\Phi_{i,\K}\bigl(s^0_1+j_1a, \dots, s^0_{i-1} + j_{i-1}a,
s^0_i +d_ia, s^0_{i+1}, \dots, s^0_r,x^0, a\bigr),
\eean
if nonzero, is an eigenvector of $\tilde K_{m,\K}(x^0,a)$ with eigenvalue $\tilde \Psi_{m,\K}(s^0,x^0,a)$:
\[
\tilde K_{m,\K}(x^0,a)\,J(s^0,x^0,a)=\tilde \Psi_{m,\K}(s^0,x^0,a)\,J(s^0,x^0,a),
\qquad m=1,\dots,n,
\]
see Corollary \ref{cor m eigen}.

\vsk.2>
Formula \eqref{mx00} reveals a fundamental difference from the classical Bethe ansatz theory.
The eigenvector of commuting operators $(K_{m,\K}(x,0))$ is obtained
simply by evaluating the weight function at a Bethe solution.
In characteristic $p$, the eigenvector of the $p$-curvature operators
$(\tilde K_{m,\K}(x,a))$ is instead a discrete hypergeometric,
or discrete Gauss--Manin, sum over the finite lattice
\bea
\Z^r/p\Z^r,
\eea
formed from shifted values of the weight function, with scalar coefficients
determined by the discrete parallel transport of the associated rank-one connection.
Thus the $p$-curvature eigenvectors are constructed as a finite, discrete analogue of a
hypergeometric integral.

\vsk.2>

The scalar functions $\tilde \Phi_{l,\K}(s,x,a)$ and $\tilde \Psi_{m,\K}(s,x,a)$ satisfy a special symmetry:
\bean
\label{SS}
\tilde \Phi_{l,\K}(s,x,a)
&=&
\Phi_{l,\K}(h_p(s_1), \dots, h_p(s_r), h_p(x_1),\dots,h_p(x_n),a),
\\
\notag
\tilde \Psi_{m,\K}(s,x,a)
&=&
\Psi_{m,\K}(h_p(s_1), \dots, h_p(s_r), h_p(x_1),\dots,h_p(x_n),a),
\eean
where
\bea
h_p(y) = y^p-a^{p-1}y
\eea
is the
Artin--Schreier polynomial, and some parameters in $\Phi_{l,\K}$ and $\Psi_{m,\K}$ in \eqref{SS} are also shifted by the transformations $y\mapsto h_p(y)$ and $y\mapsto y^p$.
See Lemma \ref{lem tilde} for the precise statement on the symmetry \eqref{SS} and the Motivating Example in Section~\ref{n=2 p=3}.

\vsk.2>

Thus the Bethe ansatz picture for $a=0$ and the $p$-curvature picture in characteristic $p$
are closely parallel. In the case $a=0$, solving
\[
\Phi_{l,\K}(s,x,0)=1,
\qquad l=1,\dots,r,
\]
produces an eigenvector $W_\K(s^0,x^0)$ of the operators $(K_{m,\K}(x^0,0))$ with eigenvalues
\\
$(\Psi_{m,\K}(s^0,x^0,0))$.
In characteristic $p$, instead solving
\[
\Phi_{l,\K}(h_p(s_1),\dots,h_p(s_r), h_p(x_1), \dots, h_p(x_n),a)=1,
\qquad l=1,\dots,r,
\]
with appropriately modified parameters,
produces an eigenvector $J(s^0,x^0,a)$ of the $p$-curvature operators
$\tilde K_{m,\K}(x^0,a)$, $m=1,\dots,n$, with eigenvalues
\[
\bigl(\Psi_{m,\K}(h_p(s_1^0),\dots,h_p(s_r^0), h_p(x_1^0), \dots, h_p(x_n^0),a)\bigr).
\]

\subsection{Motivating Example}
\label{n=2 p=3}\rm

Let \,$M_{\La_1,\K}\ox M_{\La_2,\K}$ \,be the tensor product of two Verma modules over
the Lie algebra \,$\gl_2$.  Let $V_\K\subset M_{\La_1,\K}\ox M_{\La_2,\K}$ be the two-dimensional subspace
with the basis
\be
f^{(1,0)}=e_{21}\<\>v_{\La_1}\ox v_{\La_2}\>,\qquad
f^{(0,1)}=v_{\La_1}\ox e_{21}\<\>v_{\La_2}\>.
\ee
In this basis, the rational $R$-matrices on $V_\K$ are
\be
R^{(1,2)}_{\La_1,\La_2}(x)\,=\,\frac1{x-\La_1}\,
\begin{pmatrix} x-\La_1\<+\La_2 & -\<\>\La_2\\[4pt]
-\La_1 & x \end{pmatrix}
\ee
and
\be
R^{(2,1)}_{\La_2,\La_1}(x)\,=\,\frac1{x-\La_2}\,
\begin{pmatrix} x & -\<\>\La_2\\[4pt]
-\La_1 & x+\La_1\<-\La_2 \end{pmatrix}.
\ee

\smallskip
\noindent
The discrete qKZ connection is given by the operators
\begin{align*}
K_{1,\K}(x_1,x_2,a)\, &{}=\,\ka^{e_{22}^{(1)}}R^{(1,2)}_{\La_1,\La_2}(x_1-x_2)\,,
\\
K_{2,\K}(x_1,x_2,a)\, &{}=\,R^{(2,1)}_{\La_2,\La_1}(x_2\<-x_1\<+a)\>\ka^{e_{22}^{(2)}}\,,
\end{align*}
where
\be
\ka^{e_{22}^{(1)}}\,=\,\begin{pmatrix} \,\ka & 0\,\\[2pt] \,0 & 1\,\end{pmatrix},\qquad
\ka^{e_{22}^{(2)}}\,=\,\begin{pmatrix} \,1 & 0\,\\[2pt] \,0 & \ka\,\end{pmatrix}.
\ee
In this example $s$ is a single variable, and
\begin{align*}
\Phi_{\K}(s,x_1,x_2,a)\, &{}=\,
\ka\;\frac{(s-x_1\<+\La_1)\>(s-x_2\<+\La_2)}{(s-x_1\<+a)\>(s-x_2\<+a)}\;,
\\[4pt]
\Psi_{m,\K}(s,x_1,x_2,a)\, &{}=\;\frac{s-x_m}{s-x_m\<+\La_m\<-a}\,,\qquad m\,=\,1,2\,.
\end{align*}
The $V_\K$-valued weight function is
\be
W_\K(s,x_1,x_2)\,=\,(s-x_2)\> f^{(1,0)} +\>(s-x_1\<+\La_1)\>f^{(0,1)}\>.
\ee
Let \,$p=3$\,. Then the \,$p$-curvature operators are
\begin{align*}
\tilde K_{1,\K}(x_1,x_2,a)\,&{}=\,K_{1,\K}(x_1\<+2\<\>a,x_2,a)\,K_{1,\K}(x_1\<+a,x_2,a)\,K_{1,\K}(x_1,x_2,a)\,,
\\
\tilde K_{2,\K}(x_1,x_2,a)\,&{}=\,K_{2,\K}(x_1,x_2\<+2\<\>a,a)\,K_{2,\K}(x_1,x_2\<+a,a)\,K_{2,\K}(x_1,x_2,a)\,,
\end{align*}
and
\begin{align*}
& \tilde\Phi_{\K}(s,x_1,x_2,a)\,=\,
\ka^3\,\frac{\bigl(h_3(s)-h_3(x_1)+h_3(\La_1)\bigr)\>\bigl(h_3(s)-h_3(x_2)+h_3(\La_2)\bigr)}
{\bigl(h_3(s)-h_3(x_1)\bigr)\>\bigl(h_3(s)-h_3(x_2)\bigr)}\;,
\\[4pt]
& \tilde\Psi_{m,\K}(s,x_1,x_2,a)\,=\;\frac{h_3(s)-h_3(x_m)}{h_3(s)-h_3(x_m)+h_3(\La_m)}\;,
\qquad m\,=\,1,2\,,
\end{align*}
where \,$h_3(y)=y^3\<-a^2y\,$.

\smallskip

Corollary \ref{cor m eigen} in this case reads as follows.
Let \,$(s^0, x_1^0,x_2^0)$ \,be a solution to the Bethe ansatz equation
\be
\ka^3\,\frac{\bigl(h_3(s)-h_3(x_1)+h_3(\La_1)\bigr)\>\bigl(h_3(s)-h_3(x_2)+h_3(\La_2)\bigr)}
{\bigl(h_3(s)-h_3(x_1)\bigr)\>\bigl(h_3(s)-h_3(x_2)\bigr)}\;=\,1\,.
\ee
Then the vector
\begin{align*}
J(s^0, x_1^0,x_2^0,a)\,&{}=\,W_\K(s^0,x_1^0,x_2^0)\>+\>
W_\K(s^0\<+a,x_1^0,x_2^0)\,\Phi_\K(s^0,x_1^0,x_2^0,a)\>+{}
\\[3pt]
&{}\>+\,W_\K(s^0\<+2\<\>a,x_1^0,x_2^0)\,\Phi_\K(s^0,x_1^0,x_2^0,a)\,\Phi_\K(s^0\<+a,x_1^0,x_2^0,a),
\end{align*}
if nonzero,
is an eigenvector of the \,$p$-curvature operators \,$\tilde K_{m,\K}(x_1^0,x_2^0,a)$\,,
\,$m=1,2$\,, with respective eigenvalues
\be
\frac{h_3(s^0)-h_3(x_m^0)}{h_3(s^0)-h_3(x_m^0)+h_3(\La_m)}\;.
\ee

\medskip
Our construction raises the usual Bethe-ansatz-type questions:
\begin{itemize}
\item
Is the constructed vector nonzero?
\item Do all eigenvectors of the $p$-curvature operators arise in this way?
\item Can one compute their norms?
\item Are the resulting eigenvectors orthogonal?
\end{itemize}

\subsection{Invariant discrete connections}

This paper is closely related to \cite{TV3}. Both papers are based on the notion of an invariant discrete connection and its integral representation. An invariant discrete connection carries commuting monodromy operators, while an integral representation of the connection provides a construction of joint eigenvectors of these operators.

The notion of an invariant discrete connection was introduced and studied in \cite{TV3}. The present paper develops its analogue in finite characteristic.

In \cite{TV3}, the construction was applied to the monodromy operators associated
with multiplicative qKZ-type connections whose multiplicative step is a root of unity.
Here we apply a finite-characteristic modification of the same construction to the
$p$-curvature operators.

In \cite{TV4}, we apply similar ideas to the
differential equations in characteristic $p$ with integral representations.
We construct eigenvectors and eigenvalues of the associated $p$-curvature operators.

\medskip

This paper is organized as follows.
\begin{itemize}
\item \textbf{Sections \ref{sec 2} and \ref{sec 3}:} We consider
a $p$-invariant discrete flat connection on the trivial bundle
$\Z^{n}\times V_\K \to \Z^{n}$, introduce the notion of its integral representation, and
construct flat eigensections of the $p$-curvature operators.

\item \textbf{Sections \ref{sec 4} -- \ref{sec 6}:} We apply this construction to the qKZ-type
additive discrete flat connections with integral representations.
\end{itemize}

To keep the notation minimal, in this paper we focus on applications of the general construction
to the rational qKZ discrete connection associated with a tensor product of Verma modules over $\mathfrak{gl}_2$ in characteristic $p$.

\smallskip
In Section \ref{sec 7}, we illustrate the construction in special cases in which
$\La_1,\dots,\La_n\in\F_p$. In particular, we consider the case
\[
\La_{1},\dots,\La_{n}\in \F_p\subset \K,
\quad a\in\F_p^\times, \quad \ka=1.
\]
In this case, the system of equations \eqref{gbae} is trivial: the equations hold
identically as rational functions, so every regular point $(s^0,x^0)$ satisfies
the system, and our method produces polynomial flat sections of the discrete flat connection
$(K_{m,\K}(x,a))$. Each of these polynomial flat sections is an eigensection of the
$p$-curvature operators $(\tilde K_{m,\K}(x,a))$, with all eigenvalues equal to one.

This special case has a flavor similar to the construction of polynomial solutions of the
KZ and qKZ equations in characteristic $p$ presented in \cite{SV, MV1}.

\smallskip

The paper has two appendices.
In Appendix~\ref{app A}, we list elementary properties of Pochhammer polynomials.

In Appendix~\ref{appB}, the qKZ discrete connection is considered on the degree-one subspace of a tensor product of Verma modules over $\frak{gl}_2$. It is shown that, for generic values of the parameters, our construction produces a basis of common eigenvectors of the $p$-curvature operators.

\smallskip

\noindent\textbf{Acknowledgments.}
The second author thanks IH\'ES for its hospitality during May--June 2026, when this paper was developed.
The authors thank P.\,Etingof for useful discussions.

\section{Discrete flat connection}\label{sec 2}

\subsection{Definition}

Let $\K$ be a field of characteristic $p>2$ and let $V_\K$ be a vector space over $\K$.
Let $n$ be a positive integer. Let $f_1, \dots, f_n$ be the standard basis vectors of
the lattice $\Z^n$.

\vsk.2>

A collection of $\on{GL}(V_\K)$-valued
functions $L_m: \Z^n \to \on{GL}(V_\K)$ for $m=1, \dots, n$ defines a {\it $p$-invariant
discrete flat connection} on the trivial bundle $\Z^n \times V_\K \to \Z^n$ if they are
$p$-periodic,
\beq
\label{L per}
L_i(k+pf_j)
= L_i(k), \quad 1 \leq i, j \leq n, \quad k\in\Z^n,
\eeq
and satisfy the flatness conditions:
\beq
\label{L com}
L_i(k + f_j) L_j(k)
=
L_j(k + f_i) L_i(k), \quad 1 \leq i,j \leq n, \quad k\in\Z^n.
\eeq

\vsk.2>

A function $I: \Z^n \to V_\K$ is called a
{\it flat section} if it satisfies:
\[
I(k + f_i) ={L}_i(k) I(k),
\]
for all $i=1, \dots, n$ and $k\in \Z^n$.
Let $\Ga$ be the vector space of flat sections.
Evaluation at the origin of the lattice, $I \mapsto I(0)$, provides an isomorphism $\Ga\to V_\K$.

\subsection{$p$-Curvature functions}

Define the $p$-curvature functions $C_l: \Z^n \to \on{GL}(V_\K)$ for $l=1, \dots, n$ by the ordered product:
\beq
\label{c1}
C_l(k) = L_l(k + (p-1)f_l) \dots L_l(k + f_l) L_l(k).
\eeq
The operator $C_l(k)$ is the transport along $p$ consecutive steps in the $f_l$-direction.

\begin{lem}
\label{lem 2.1}
The $p$-curvature functions define automorphisms of the discrete flat connection $(L_m)$.
That is, for $1 \leq l, m \leq n$, we have:
\beq
\label{c2}
C_l(k + f_m) L_m(k) = L_m(k) C_l(k).
\eeq
Furthermore, the $p$-curvature functions mutually commute:
\beq
\label{c3}
C_l(k) C_m(k) = C_m(k) C_l(k),
\eeq
and the $p$-curvature functions are $p$-periodic:
\beq
\label{c4}
C_l(k+pf_m) = C_l(k), \qquad 1\leq l,m\leq n.
\eeq

\end{lem}

\begin{proof}

Formula
$\eqref{c2}$ follows by repeatedly applying
$$L_l(k+f_m)L_m(k)=L_m(k+f_l)L_l(k)$$
to move $L_m(k)$ through the $p$-step product. One obtains
$$C_l(k+f_m)L_m(k)
=
L_m(k)C_l(k).$$
For commutativity, compare parallel transport along the two $p\times p$ rectangles:
$$
C_l(k+pf_m)C_m(k)
=
C_m(k+pf_l)C_l(k).
$$
Using $p$-periodicity gives
$$C_l(k)C_m(k)=C_m(k)C_l(k).$$
Finally, $\eqref{c4}$ follows immediately from the $p$-periodicity of $L_l$\,.
\end{proof}

For any $I \in \Ga$, the action
\bea
I(k) \,\mapsto C_l(k)I(k)
\eea
produces another element of $\Ga$. For each $l$, this
defines an automorphism of $\Ga$.
These automorphisms commute, and our subsequent goal is to construct elements $I \in \Ga$ that are joint eigenvectors of these
automorphisms.

\section{Integral representations}
\label{sec 3}

In this section we develop the notion of an integral representation
of a $p$-invariant discrete flat connection on $\Z^n \times V_\K \to \Z^n$.

\subsection{Discrete flat connection of rank one}

Let $r$ be a positive integer.
Let $e_1, \dots, e_r$ be the standard basis vectors for $\Z^r$. We consider the lattice $\Z^{r+n}=\Z^r\times \Z^n$.

\vsk.2>

A $p$-invariant discrete flat connection on the trivial line bundle
$\Z^{r+n}\times \K \to \Z^{r+n}$ is defined by
a collection of $p$-periodic functions $\phi_l : \Z^{r+n} \to \K^\times$ for $l=1, \dots, r$ and
$\psi_m:\Z^{r+n}\to \K^\times$ for $m=1, \dots, n$ satisfying the flatness conditions:
\begin{align}
\label{phi-phi}
\phi_l(j+e_i,k)\phi_i(j,k)
&=
\phi_i(j+e_l,k)\phi_l(j,k),
&&1\leq i,l\leq r,\\
\notag
\psi_m(j,k+f_i)\psi_i(j,k)
&=
\psi_i(j,k+f_m)\psi_m(j,k),
&&1\leq i,m\leq n,\\
\notag
\psi_m(j+e_i,k)\phi_i(j,k)
&=
\phi_i(j,k+f_m)\psi_m(j,k),
&&1\leq i\leq r,\ 1\leq m\leq n,
\end{align}
where $(j,k)\in\Z^r\times\Z^n$.

The corresponding scalar $p$-curvature functions are given by:
\begin{align*}
\tilde{\phi}_l(j,k) = \prod_{i=0}^{p-1} \phi_l(j + i e_l, k),
\qquad
\tilde{\psi}_m(j,k) = \prod_{i=0}^{p-1} \psi_m(j, k + i f_m).
\end{align*}

\begin{lem}
\label{lem 2.2}
The $p$-curvature functions $(\tilde \phi_l, \tilde\psi_m)$ are constant,
that is, independent of $(j,k)$.

\end{lem}

\begin{proof}
Since the connection has rank one, the operators are scalar. Applying
\eqref{c2} in each of the $e_i$- and $f_m$-directions gives
\[
\widetilde\phi_l(j+e_i,k)=\widetilde\phi_l(j,k),\qquad
\widetilde\phi_l(j,k+f_m)=\widetilde\phi_l(j,k),
\]
and similarly
\[
\widetilde\psi_m(j+e_i,k)=\widetilde\psi_m(j,k),\qquad
\widetilde\psi_m(j,k+f_i)=\widetilde\psi_m(j,k).
\]
Hence all these functions are constant.
\end{proof}

\subsection{Parallel transport}

The transition functions $(\phi_l, \psi_m)$
define parallel transport of the fiber $\K$ of the bundle
$\Z^{r+n}\times \K\to \Z^{r+n}$\,.
Namely, the transport of fibers
$\K_{(j,k)} \leftarrow \K_{(j +e_l,k)}$ and
$\K_{(j,k)} \leftarrow \K_{(j,k+f_m)}$ are multiplications by
$\phi_l(j,k)$ and $\psi_m(j,k)$, respectively.\footnote{We use the convention that $\phi_l(j,k)$ is the transport from the fiber at $(j+e_l,k)$ to the fiber at $(j,k)$, and $\psi_m(j,k)$ is the transport from $(j,k+f_m)$ to $(j,k)$.}

These elementary transports and the flatness conditions determine a transport
$\K_{(j^1,k^1)} \leftarrow \K_{(j^2,k^2)}$ for any two points
$(j^1, k^1)$ and $(j^2, k^2)$ of $\Z^{r}\times \Z^n$
as multiplication by a uniquely determined element of $\K^\times$ which we denote by
$\mu_{(j^1, k^1),(j^2, k^2)}$.

For example, for nonnegative integers $i_1,\dots,i_r$, we have:
\bean
\label{pt}
\mu_{(0,0),(i_1e_1 + \dots + i_re_r,0)}
=
\prod_{l=1}^r\prod_{d_l=0}^{i_l-1}
\phi_l(i_1e_1+\dots +i_{l-1}e_{l-1} +d_le_l, 0),
\eean
where $(0,0)$ is the zero vector in $ \Z^r\times \Z^n$.

\subsection{Discrete integral}

Let $u : \Z^{r+n}\to V_\K$ be a $p$-periodic function.
Define a function $I : \Z^n\to V_\K$ by the formula:
\bean
\label{int}
I(k) = \sum_{i_1, \dots, i_r=0}^{p-1}
u(i_1e_1+\dots + i_re_r, k)\ \mu_{(0,0),(i_1e_1+\dots + i_re_r, k)}\,.
\eean
We call this function the {\it discrete integral} of $u$ in the $\Z^r$-direction and denote its values by
\bea
I(k) =: \int u(j, k) d j\ .
\eea

\subsection{Special discrete connection}
\label{sec spe}

The $p$-invariant discrete flat connection on $\Z^{r+n}\times \K \to \Z^{r+n}$ defined by
the transition functions
$(\phi_l,\psi_m)$ is called {\it special} if
\bean
\label{spe}
\tilde \phi_l= 1, \qquad l=1,\dots,r.
\eean
Recall that the $(\tilde \phi_l)$ are constant functions on $\Z^{r+n}$ by Lemma \ref{lem 2.2}.

\begin{lem}
\label{lem int d}
Assume 
that the rank-one discrete connection satisfies
\[
\widetilde\phi_l=1,\qquad l=1,\dots,r.
\]
Let
$u:\Z^{r+n} \to V_\K$ be a $p$-periodic function.
Then
\bean
\label{ili d}
\int \left(\phi_l(j,k)\,u(j+e_l,k) - u(j,k)\right) dj = 0
\eean
for all $k\in\Z^n$ and $l=1,\dots,r$.

\end{lem}

This lemma is a version of \cite[Lemma 3.2]{TV3}.
Lemma \ref{lem int d} is a discrete analog of Stokes' theorem.

\begin{proof}

We prove \eqref{ili d} for $l=1$. The proof for other $l$ is similar.
The key identity is:$$\mu_{(0,0),(j,k)}\,\phi_l(j,k)
=
\mu_{(0,0),(j+e_l,k)}.$$
This follows from the transport convention introduced earlier: $\phi_l(j,k)$ transports from $(j+e_l,k)$ to $(j,k)$. Since the connection is rank one, the factors commute, so the order is immaterial.

For $l=1$, we have
\bean
\label{d1}
&&
\int \phi_1(j,k)\,u(j+e_1,k) d j
=
\sum_{i_1, \dots, i_r=0}^{p-1}
\phi_1(i_1e_1+ \dots+ i_re_r,k)
\\
\notag
&&
\times \ \
u((i_1+1)e_1 + i_2e_2+ \dots +i_re_r, k)
\ \mu_{(0,0),( i_1e_1+\dots + i_re_r,k)}\,
\\
\notag
&&
=
\sum_{i_1, \dots, i_r=0}^{p-1}
u((i_1+1)e_1+i_2e_2+ \dots + i_re_r, k) \,\mu_{(0,0),((i_1+1)e_1+i_2e_2+\dots+i_re_r,k)}\,.
\eean
This sum equals
\bea
&&
\int u(j,k) dj
=
\sum_{i_1, \dots, i_r=0}^{p-1}
u(i_1e_1+i_2e_2+ \dots +i_re_r,k)
\mu_{(0,0),(i_1e_1+i_2e_2+\dots+i_re_r,k)}\,,
\eea
since for any $i_2, \dots, i_r$ we have
\bea
&&
u(pe_1+i_2e_2+ \dots+i_re_r,k)
=
u(i_2e_2+\dots+i_re_r, k),
\eea
because $u$ is $p$-periodic,
and
\bea
\mu_{(0,0),(pe_1+i_2e_2+\dots+i_re_r,k)}
&=&
\mu_{(0,0),(i_2e_2+\dots + i_re_r,k)}\, \tilde \phi_1(i_2e_2+\dots +i_re_r, k)
\\
&= &
\mu_{(0,0),(i_2e_2+\dots+i_re_r,k)} ,
\eea
due
to the assumption that $\tilde \phi_1=1$.
The lemma is proved.
\end{proof}

\subsection{Integral representation}
\label{sec 2.4}

Assume that $\on{GL}(V_\K)$-valued functions $(L_m)$
define a $p$-invariant discrete flat connection
on $\Z^n\times V_\K\to \Z^n$.
Assume
that $\K^\times$-valued functions $(\phi_l,\psi_m )$ define
a $p$-invariant discrete flat connection on $\Z^{r+n}\times \K \to \Z^{r+n}$.

\smallskip

Let $w$ and $g_{l,m}$, $l=1, \dots, r$, $m=1,\dots,n$, be $V_\K$-valued $p$-periodic functions
on $\Z^{r+n}$. Assume that these functions satisfy the relations:
\bean
\label{IR}
\phantom{aaaaaa}
\psi_m(j,k)\, w(j,k + f_m) \,=\, L_m(k)w(j,k) \,+\,
\sum_{l=1}^r \left(\phi_l(j,k)\,g_{l,m}(j+e_l, k) - g_{l,m}(j,k)\right)
\eean
for all $j\in\Z^r$, $k\in\Z^n$, $m=1,\dots,n$.
Notice that the summation terms are discrete coboundaries in the auxiliary $j$-directions.

Such a
collection of functions is called an {\it integral representation} for the $p$-invariant
discrete flat connection $(L_m)$. The function $w$ is called the {\it weight function.}

\subsection{Flat sections}

Let $w$ be the weight function of
an integral representation. Define a $V_\K$-valued function
$I : \Z^n \to V_\K,$
by the formula:
\bean
\label{ix1}
\phantom{aaa}
I(k)
&=&
\int w(j,k) dj
\\
\notag
&=&
\sum_{i_1, \dots, i_r=0}^{p-1}
w(i_1e_1 +\dots + i_re_r,k)\ \mu_{(0,0),(i_1e_1+\dots+i_re_r, k)}\,.
\eean

\begin{thm}
\label{thm sol}

If the $p$-invariant discrete flat connection
$(\phi_l,\psi_m)$ is special, then $I$ is a flat section
of the $p$-invariant discrete flat connection $(L_m)$,
\bean
\label{ili}
I(k+f_m) = L_m(k) I(k),
\eean
for all $k\in\Z^n$ and $m=1,\dots, n$.

\end{thm}

This is a reformulation of \cite[Theorem 3.3]{TV3}.
\begin{proof}

We prove relation \eqref{ili} for $m=1$. The proof for other $m$ is similar.
We have
\bean
\label{IR1}
\phantom{aaa}
\psi_1(j,k) \,w(j,k+f_1) \,= \,L_1(k)\,w(j,k) +
\sum_{l=1}^r \left(\phi_l(j,k)\,g_{l,1}(j+e_l,k) - g_{l,1}(j,k)\right)
\eean
by assumption. We send each term $T(j,k)$ of this identity
to the element $\int T(j,k) dj\in V_\K$.
In particular, the left-hand side is sent to
\bea
&&
\int \psi_1(j,k) \,w(j, k+f_1) dj
=
\sum_{i_1, \dots, i_r=0}^{p-1}
w(i_1e_1+ \dots +i_re_r, k+f_1)
\\
&&
\times\ \
\psi_1(i_1e_1+\dots+ i_re_r,k)\,
\mu_{(0,0),(i_1e_1+\dots+ i_re_r,k)}
\\
&&
=
\sum_{i_1, \dots, i_r=0}^{p-1}
w(i_1e_1+\dots+ i_re_r,k+f_1)\,
\mu_{(0,0),(i_1e_1+\dots+ i_re_r,k+f_1)}
=
I(k+f_1).
\eea
The first term on the right-hand side is sent to
\bea
\int L_1(k)w(j,k) dj =
L_1(k)\int w(j,k) dj
=
L_1(k)\,I(k).
\eea
Each summand in the final sum in  \eqref{IR1} vanishes by
Lemma \ref{lem int d}. Relation
\eqref{ili} is proved for $m=1$. Theorem \ref{thm sol} is proved.
\end{proof}

\subsection{Eigensections and eigenvalues}

\begin{thm}
\label{thm i sol}

The flat section $I : \Z^n \to V_\K$ of Theorem \ref{thm sol}, if nonzero, is an eigensection of the $p$-curvature functions:
\bean
\label{e so}
C_m(k)I(k) = \tilde \psi_m(0,0)I(k),
\eean
for $k\in\Z^n$ and $m=1,\dots,n$.

\end{thm}

Recall that $ \tilde \psi_m(0,0)$ is the value
at $(0,0)$ of the $p$-curvature function
$ \tilde \psi_m :\Z^{r+n}\to \K^\times$ which is constant
on $\Z^{r+n}$.

Theorem \ref{thm i sol}
is a reformulation of \cite[Theorem 3.4]{TV3}.

\begin{proof}

We have $C_m(k)I(k)= I(k +pf_m)$ and
\bea
I(k +pf_m)
&=&
\sum_{i_1, \dots, i_r=0}^{p-1}
w(i_1e_1+\dots + i_re_r, k+pf_m)\, \mu_{(0,0),(i_1e_1+\dots +i_re_r,k+pf_m)}.
\eea
We have
\bea
w(i_1e_1+\dots + i_re_r, k+pf_m)
=
w(i_1e_1+\dots+ i_re_r, k)
\eea
due to $p$-periodicity of $w$.
We also have
\bea
\mu_{(0,0),(i_1e_1+\dots + i_re_r, k+pf_m)}
&=&
\mu_{(0,0),(i_1e_1+\dots+ i_re_r, k)}\,
\tilde \psi_m (i_1e_1+\dots +i_re_r, k)
\\
&=&
\mu_{(0,0),(i_1e_1+\dots +i_re_r, k)}\,
\tilde \psi_m (0,0),
\eea
since $\tilde \psi_m$ is constant on $\Z^{r+n}$.
The theorem is proved.
\end{proof}

\begin{cor}
\label{cor eigen}

Assume that the $p$-invariant discrete flat connection $(\phi_l,\psi_m)$ is special.
Then the vector
\bean
\label{ix0}
\phantom{aaa}
I(0)
&=&
\int w(j,0)\, dj
\\
\notag
&=&
\sum_{i_1, \dots, i_r=0}^{p-1}
w(i_1e_1+\dots + i_re_r,0)\ \mu_{(0,0),(i_1e_1+\dots + i_re_r,0)}
\eean
if nonzero, 
is an eigenvector of the $p$-curvature operators
$\left(C_m(0)\right)$ with respective eigenvalues $\left(\tilde \psi_m(0,0)\right)$.

\end{cor}

The coefficients $\mu_{(0,0),(i_1e_1+\dots + i_re_r,0)}$
are given by formula \eqref{pt}.

\begin{rem}

Given a $p$-invariant discrete flat connection $(L_m)$, one can construct flat eigensections of this connection whenever an integral representation of the type described in Section \ref{sec 3} is available.
Integral representations are known to exist for systems of rational qKZ difference
equations and for systems of qKZ difference equations associated with equivariant cohomology of Nakajima varieties.
If such a system is considered over a field of finite characteristic,
then the construction of Section \ref{sec 3} applies, and the corresponding flat eigensections can be obtained.

\end{rem}

\section{Discrete qKZ connection over $\Z$}
\label{sec 4}

\subsection{Verma modules}
\label{sec 4.1}

Consider the Lie algebra \,$\gl_2$ \,over $\Z$ with the standard generators
\,$e_{11}\,,\;e_{12}\,,\;e_{21}\,,\;e_{22}$\,. Set
\be
e=e_{12}\>,\quad f=e_{21}\>,\quad h_1=e_{11}\>,\quad h_2=e_{22}\>.
\ee

Let $\La$ be a parameter, and let $M_\La$ denote the Verma module over
$\mathfrak{gl}_2$ with highest weight \,$[\<\>\La,0\<\>]$ \,and highest weight vector $v_\La$:
\be
e\<\>v_\La=0,\qquad h_1\<\>v_\La=\La\>v_\La,\quad h_2\<\>v_\La=\<\>0\,.
\ee
A basis of $M_\La$ is formed by the vectors $f^{\<\>r}v_\La$, $r\in\Z_{\geq 0}$.

Let $\La_1,\dots,\La_n$ be parameters.
We have the weight decomposition
\bea
\tsize\bigotimes^n_{j=1} M_{{\La_j}} =
\bigoplus_{r=0}^\infty
\bigotimes^n_{j=1} M_{{\La_j}}\left[\<\>\sum_{j=1}^n\La_j-r,r\<\>\right]\,.
\eea
The basis of a weight subspace
\,$V=\left(\ox^n_{j=1} M_{{\La_j}}\right)\left[\<\>\sum_{j=1}^n\La_j-r,r\<\>\right]$
\,is formed by the vectors
\beq
\label{basis}
f^{\<\>\vec r}:= f^{\<\>r_1}v_{\La_1} \ox\dots\ox f^{\<\>r_n}v_{\La_n}\,,
\eeq
where $\vec r=(r_1,\dots,r_n)$, $r_1+\dots+r_n = r$.
The set of such indices $\vec r$ is denoted by $\mc I_r$.

\smallskip

The kernel of the map, induced by the action of
$e=e_{12}$,
\bea
\tsize
e\, :\, \left(\bigotimes^n_{j=1} M_{{\La_j}}\right)\left[\<\>\sum_{j=1}^n\La_j-r,r\<\>\right]
\to \left(\bigotimes^n_{j=1} M_{{\La_j}}\right)\left[\<\>\sum_{j=1}^n\La_j-r+1,r-1\<\>\right],
\\
u\mapsto e\,u,
\eea
is denoted by
$\Sing \left(\bigotimes^n_{j=1} M_{{\La_j}}\right)\left[\<\>\sum_{j=1}^n\La_j-r,r\<\>\right]$ and called the subspace of singular vectors of weight $\left[\<\>\sum_{j=1}^n\La_j-r,r\<\>\right]$.

\subsection{Rational $R$-matrix}

The rational $\mathfrak{gl}_2$ $R$-matrix
\[
R_{\La_1,\La_2}(x)
\in
\End(M_{\La_1}\otimes M_{\La_2})
\otimes
\mathbb Z(x,\La_1,\La_2)
\]
is the unique operator satisfying the relations
\begin{align}
\label{R1}
R_{\La_1, \La_2}(x)\>(\<\>f\ox 1+1\ox f\<\>)\,&{}=\,
(\<\>f \ox 1+1\ox f\<\>)\>R_{\La_1, \La_2}(x)\,,
\\[4pt]
\label{R2}R_{\La_1, \La_2}(x)\>(\<\>f \ox h_2+(h_1\<-x)\ox f\<\>)\,&{}=\,
(\<\>f \ox h_1+(h_2\<-x)\<\ox f\<\>)\>R_{\La_1, \La_2}(x)\,,
\\[4pt]
\label{R3}
R_{\La_1, \La_2}(x)\>v_{\La_1}\ox v_{\La_2}&{}=\,v_{\La_1}\ox v_{\La_2}\,,
\end{align}
see, for example, \cite{TV1, TV2}.

It is known that the $R$-matrix
preserves the weight decomposition and the subspaces of singular vectors for the diagonal $\mathfrak{gl}_2$-action.

For the restrictions of \,$R_{\La_1,
\La_2}(x)$ \,and \,$\bigl(R_{\La_1,\La_2}(x)\bigr)^{-1}$ to the weight
subspace $\bigl(M_{\La_1}\!\ox M_{\La_2}\bigr)\<\>[\<\>\La_1\<+\La_2\<-r,r\<\>]$, their entries have respectively the form
\be
P(x,\La_1,\La_2)\,\prod_{i=0}^{r-1}\,(x-\La_1\<+i)^{-1}
\quad\text{and}\quad\,
Q(x,\La_1,\La_2)\,
\prod_{i=0}^{r-1}\,(x+\La_2\<-i)^{-1}\>,
\ee
where \,$P(x,\La_1,\La_2)$ \,and \,$Q(x,\La_1,\La_2)$ \,are polynomials
in \,$x,\La_1,\La_2$ \,with integer coefficients.

\begin{exmp}
The restriction of the rational $R$-matrix to the weight subspace
$\bigl(M_{\La_1}\!\ox M_{\La_2}\bigr)\<\>[\<\>\La_1\<+\La_2\<-1,1\<\>]$
\,with the basis \,$fv_{\La_1}\!\ox v_{\La_2}\,,\;v_{\La_1}\!\ox fv_{\La_2}$
equals
\be
R_{\La_1,\La_2}(x)\,=\,\frac1{x-\La_1}\,
\begin{pmatrix} x-\La_1\<+\La_2 & -\<\>\La_2\\[4pt]
-\La_1 & x \end{pmatrix}.
\ee
\end{exmp}

\subsection{Rational qKZ connection}

Let $a, \ka$, $\La=(\La_1$, \dots, $\La_n)$ and $x=(x_1,\dots,x_n)$ be variables.
Fix a weight subspace
\bea
\tsize V=\left(\bigotimes^n_{j=1} M_{{\La_j}}\right)\left[\<\>\sum_{j=1}^n\La_j-r,r\<\>\right]\,.
\eea
The qKZ functions with fiber $V$ and step $a$ are the following $\on{GL}(V)$-valued rational functions in
$a, \ka, \La , x$\ :
\bean
\label{Kmx}
K_m(x)
&=&
R_{{\La_{m}}, {\La_{m-1}}}^{(m,m-1)}(x_m\<-x_{m-1}\<+a)\dots
R_{{\La_{m}}, {\La_{1}}}^{(m,1)}(x_m\<-x_1\<+a)\times{}
\\[3pt]
\notag
&\times &
\ka^{h_2^{(m)}}R_{{\La_{m}}, {\La_{n}}}^{(m,n)}(x_m\<-x_n)\dots
R_{{\La_{m}}, {\La_{m+1}}}^{(m,m+1)}(x_m\<-x_{m+1})\,,
\eean
where $m=1,\ldots,n$\,.

\smallskip
The qKZ functions $(K_m)$ define a {\it flat discrete qKZ connection with fiber $V$ and step $a$},
that is,
\bean
\label{f qKZ}
K_m(x+af_l) K_l(x) = K_l(x+ af_m) K_m(x), \qquad 1\leq l,m\leq n,
\eean
where $x+af_i=(x_1,\dots, x_{i-1}, x_i +a, x_{i+1}, \dots, x_n)$ for $i=1,\dots,n$.

\smallskip
If $\ka=1$, then the qKZ functions preserve each of the singular vector subspaces
\\
$\Sing \left(\bigotimes^n_{j=1} M_{{\La_j}}\right)\left[\<\>\sum_{j=1}^n\La_j-r,r\<\>\right]$.

\subsection{Discrete flat connection of rank one}

Let $s=(s_1,\dots,s_r)$ be variables. Define

\bean
\label{6a.1}
\Phi_l(s\<\>,x) &=&
\ka\;\prod_{i=1}^n\,\frac{s_l-x_i+\La_i}{s_l-x_i+a}\,\cdot\<\>
\prod_{j\ne l}\,\frac{(s_l-s_j+a)\>(s_l-s_j-1)}{(s_l-s_j)\>(s_l-s_j+1+a)}\;,
\\
\notag
\Psi_m(s\<\>,x) &=& \prod_{j=1}^r\,\frac{s_j-x_m}{s_j-x_m+\La_m-a}\;,
\eean
$l=1,\dots,r$\,, \,$m=1,\dots,n$\,. These are rational functions in $a, \ka, \La, s, x$.

\vsk.2>

The functions $(\Phi_l, \Psi_m)$ satisfy the flatness relations:
\begin{align}
\label{6a.2}
\Phi_l(s + ae_i, x)\, \Phi_i(s,x) &= \Phi_i(s + ae_l, x) \,\Phi_l(s,x),
\\
\notag
\Psi_m(s, x + af_i) \,\Psi_i(s,x) &= \Psi_i(s, x + af_m) \,\Psi_m(s,x),
\\
\notag
\Psi_m(s + ae_i, x) \,\Phi_i(s,x) &= \Phi_i(s, x + af_m)\, \Psi_m(s,x),
\end{align}
where the first, second, and third relations hold respectively for
$1\leq i,l\leq r$, $1\leq i,m\leq n$, and
$1\leq i \leq r$, $1\leq m\leq n$.
Here $s + ae_i$ denotes
$(s_1,\dots, s_{i-1}, s_i+a, s_{i+1}, \dots, s_r)$, and $x+af_i$ denotes
$(x_1,\dots,x_{i-1}, x_i + a, x_{i+1}, \dots, x_n)$.

\subsection{Weight function}

The $V$-valued weight function is defined by the formula:
\bean
\label{wf1}
W(s\<\>,x)\,=\>\sum_{\vec r\in \mc I_r} W_{\vec r}(s\<\>,x)\,f^{\<\>\vec r},
\eean
\begin{align}
\label{wf2}
W_{\vec r}(s,x)
={}&
\frac{1}{r_1!\cdots r_n!}
\operatorname{Sym}_{s_1,\dots,s_r}
\Biggl[
\prod_{i=1}^n
\prod_{j=\vec r^{(i-1)}+1}^{\vec r^{(i)}}
\left(
\prod_{i'<i}(s_j-x_{i'}+\La_{i'})
\prod_{i''>i}(s_j-x_{i''})
\right)\\
\notag
&\hspace{5cm}\times
\prod_{j'<j''}
\frac{s_{j'}-s_{j''}+1}{s_{j'}-s_{j''}}
\Biggr].
\end{align}
Here we use the following notation. For a function $f(s,x)$, define
\be
\on{Sym}_{s_1,\dots,s_r}f(s,x) =
\sum_{\si\in S_r} f(s_{\si(1)}, \dots,s_{\si(r)}, x)\,.
\ee
For a vector $\vec r = (r_1,\dots,r_n)$\,, \;$r_1+\dots+r_n = r$, set
\,$\vec r^{(0)}\!=0\,$ and \,$\vec r^{(i)}\!=r_1+\dots + r_i$\,, \;$i=1,\dots,n$\,.

\medskip

The functions \,$W_{\vec r}(s\<\>,x)$ \,are polynomials in variables
\,$s,\,x,\,\La$ \,with integer coefficients.
The factorials in the denominator cancel out by the symmetrization over
the variables \,$s_j$ ``attached'' to the same \,$x_i$\,. Equivalently,
the \;$\on{Sym}_{s_1,\dots,s_r}$ \,can be replaced by the sum over cosets
\,$S_r/(S_{r_1}\!\times\dots\times S_{r_n})$ \,with no factorials in the denominator.

\subsection{Integral representation}
\label{sec 4.6}

Recall that $(\Phi_l(s,x), \Psi_m(s,x))$ are rational functions in variables $s,x,a,\La,\ka$ with integer coefficients. Recall that
$W(s,x)$ is a $V$-valued polynomial in variables $s,x,\La$ with integer coefficients,
and $(K_m(x))$ are $\on{GL}(V)$-valued rational functions
in $x,a,\La,\ka$ with integer coefficients.

\begin{thm}
[\cite{TV1, TV2}]
\label{thm int rep}
There exist \(V\)-valued rational functions \(G_{l,m}(s,x)\)
for $l=1,\dots,r$ and $m=1,\dots,n$ in variables $s, x, a, \La, \ka$, such that
\bean
\label{IR qkz}
&&
\Psi_m(s,x)\, W(s,x + af_m)
\\
\notag
&&
\phantom{aaaa}
\,=\, K_m(x)W(s,x) \,+\,
\sum_{l=1}^r \left(\Phi_l(s,x)\,G_{l,m}(s+ae_l, x) - G_{l,m}(s,x)\right)
\eean
for $m=1,\dots,n$. Furthermore, the coordinate functions of every $G_{l,m}(s,x)$ in the basis $(f^{\vec r})$
are ratios $\frac{P(s,x)}{Q(s,x)}$\,,
where $P(s,x)$ is a polynomial in $s,x, a,\La, \ka$
with integer coefficients,
and $Q(s,x)$ is a product of factors each of which is of the form:
\begin{align*}
s_k-x_i+ca\,, &\quad k=1,\dots, r, \ i=1,\dots,n,
\\
s_k-x_i + \La_i +ca\,, &\quad k=1,\dots, r, \ i=1,\dots,n,
\\
s_k-s_j+ca \,, &\quad k,j = 1,\dots, r,\ k\ne j,
\\
s_k-s_j + 1+ca\,, &\quad k,j = 1,\dots, r,\ k\ne j,
\\
x_i-x_m-\La_i+b+ca, &\quad i,m=1,\dots,n, \ i\ne m,
\\
\ka,&
\end{align*}
where $c\in \Z$, \ $b= 0, \dots, r-1$.\footnote{${}$ These factors (except $ x_i-x_m-\La_i+b+ca$ and $\ka$)
are
the factors of the denominators of the functions
$(\Phi_l, \Psi_m)$ in which variables $s$ and $x$ are shifted by integer multiples of $a$.}

\end{thm}

Recall that $V=\left(\bigotimes^n_{j=1} M_{{\La_j}}\right)\left[\<\>\sum_{j=1}^n\La_j-r,r\<\>\right]\,.$

\begin{thm}
[\cite{TV1, TV2}]
\label{thm sing}
Set $\ka=1$.
Then there exist
\\
$\left(\bigotimes^n_{j=1} M_{{\La_j}}\right)\left[\<\>\sum_{j=1}^n\La_j-r+1,r-1\<\>\right]$-valued
functions $H_l(s,x)$ for $l=1,\dots,r$, rational in variables $s,x,a,\La$, such that
\bean
\label{Si}
e\, W(s,x) \,= \,
\sum_{l=1}^r \left(\Phi_l(s,x)\,H_{l}(s+ae_l, x) - H_{l}(s,x)\right).
\eean
Furthermore, the coordinate functions of every $H_{l}(s,x)$
in the basis
$\{f^{\vec u}\}_{\vec u\in\mathcal I_{r-1}}$,
$u_1+\cdots+u_n=r-1,$
of
$\left(\bigotimes^n_{j=1} M_{{\La_j}}\right)\left[\<\>\sum_{j=1}^n\La_j-r+1,r-1\<\>\right]$
are ratios
$\frac{P(s,x)}{Q(s,x)}$\,,
where $P(s,x)$ is a polynomial in $s,x, a,\La_1,\dots,\La_n$
with integer coefficients,
and $Q(s,x)$ is a product with factors each of which is of the form
$s_i-s_j$, where $i,j=1,\dots, r$, $i\ne j$.

\end{thm}

\section{Reduction to finite characteristic}
\label{sec 5}

Let $p$ be a prime integer.
The information collected in Sections \ref{sec 4.1} -- \ref{sec 4.6} shows that all objects considered in
these sections can be reduced to characteristic $p$. As a result we obtain the following setting.

\subsection{Verma modules over $\K$}

Let $\K$ be a field of characteristic $p>2$.
Consider the Lie algebra \,$\gl_2$ \,over $\K$ with the standard generators
\,$e_{11}\,,\;e_{12}\,,\;e_{21}\,,\;e_{22}$\,. Set
\be
e=e_{12}\>,\quad f=e_{21}\>,\quad h_1=e_{11}\>,\quad h_2=e_{22}\>.
\ee
Let $\La\in\K$. Let $M_{\La,\K}$ denote the Verma module over \,$\gl_2$
\,with highest weight \,$[\La,0]$ \,and highest weight vector $v_\La$:
\be
e\<\>v_\La=0,\qquad h_1\<\>v_\La=\La\>v_\La,\quad h_2\<\>v_\La=\<\>0\,.
\ee
A basis of $M_{\La,\K}$ is formed by the vectors $f^{\<\>r}v_\La$, $r\in\Z_{\geq 0}$.

Assign degree $r$ to the basis vector $f^{\<\>r}v_\La$ of $M_{\La,\K}$.
Let $\La_1,\dots, \La_n \in\K$.
The tensor product of Verma modules over $\K$ has the degree decomposition
\bea
\tsize\bigotimes^n_{j=1} M_{\La_j,\K} =
\bigoplus_{r=0}^\infty
\left(\bigotimes^n_{j=1} M_{\La_j,\K}\right)_r\,,
\eea
where the subscript $r$ denotes the homogeneous component of total degree $r$.
Fix the degree-$r$ subspace
\,$V_\K=\left(\ox^n_{j=1} M_{\La_j,\K}\right)_r$.
A basis of $V_\K$ is formed by the vectors
\beq
\label{basisK}
f^{\<\>\vec r}= f^{\<\>r_1}v_{\La_1} \ox\dots\ox f^{\<\>r_n}v_{\La_n}\,,
\eeq
where $\vec r \in\mc I_r$.
Every vector in $V_\K$ has weight $[\sum_{j=1}^n\La_j-r,r]$;
in characteristic $p$, different degrees can have the same weight.

\subsection{Rational $R$-matrix over $\K$}

The \,$\gl_2$ rational $R$-matrix
\bea
R_{\La_1, \La_2}(x)\in \End \left( M_{\La_1,\K}\ox M_{\La_2,\K}\right)(x)
\eea
is defined by relations \eqref{R1} -- \eqref{R3}. The $R$-matrix preserves
the degree decomposition of $M_{{\La_1,\K}}\ox M_{{\La_2,\K}}$.
For the restrictions of \,$R_{\La_1,
\La_2}(x)$ \,and \,$\bigl(R_{\La_1,\La_2}(x)\bigr)^{-1}$ to the degree-$r$ subspace
$V_\K=\bigl(M_{\La_1,\K}\!\ox M_{\La_2,\K}\bigr)_r$,
each of their matrix entries has respectively the form
\be
P(x)\,\prod_{i=0}^{r-1}\,(x-\La_1\<+i)^{-1}
\quad\text{and}\quad\,
Q(x)\,
\prod_{i=0}^{r-1}\,(x+\La_2\<-i)^{-1}\>,
\ee
where \,$P(x)\,,\,Q(x)\in\K[x]$.

\subsection{Rational qKZ connection over $\K$}

Let $a, \ka \in \K^\times$,\ $\La_1$, \dots, $\La_n\in \K$. Let $x=(x_1,\dots,x_n)$ be variables.
Fix a degree-$r$ subspace
\bea
\tsize V_\K=\left(\bigotimes^n_{j=1} M_{{\La_j,\K}}\right)_r\,.
\eea
The associated qKZ functions $(K_{m,\K})$ are the $\on{GL}(V_\K)$-valued rational functions in
$x$ obtained from formula $\eqref{Kmx}$ by replacing the $R$-matrices and modules by their reductions over $\K$.

\smallskip
The qKZ functions $(K_{m,\K})$ define a flat discrete qKZ connection with fiber $V_\K$ and step $a$,
that is,
\bean
\label{ff qKZ}
K_{m,\K}(x+af_l) K_{l,\K}(x) = K_{l,\K}(x+ af_m) K_{m,\K}(x), \qquad 1\leq l,m\leq n.
\eean

\subsection{$p$-Curvature functions}

The $p$-curvature functions $\tilde K_{m,\K}$ for $m=1,\dots,n$, are the
$\on{GL}(V_\K)$-valued rational functions in $x$ defined by the formula:
\bean
\label{cqkz}
\tilde K_{m,\K}(x) = K_{m,\K}(x + (p-1)af_m) \dots K_{m,\K}(x + af_m) K_{m,\K}(x).
\eean

\begin{lem}
\label{lem 4.2}
The $p$-curvature functions define automorphisms of the discrete flat connection $(K_{m,\K})$.
That is, for $1 \leq l, m \leq n$, we have:
\[
\tilde K_{l,\K}(x + af_m) K_{m,\K}(x) = K_{m,\K}(x) \tilde K_{l,\K}(x).
\]
Furthermore, the $p$-curvature functions mutually commute:
\[
\tilde K_{l,\K}(x) \tilde K_{m,\K}(x) = \tilde K_{m,\K}(x) \tilde K_{l,\K}(x).
\]

\end{lem}

\begin{proof}
The lemma follows from the flatness conditions \eqref{ff qKZ}.
\end{proof}

\subsection{Discrete flat connection of rank one over $\K$}
\label{sec 5.5}

Let $s=(s_1,\dots,s_r)$ be variables.
Let $a, \ka \in \K^\times$,\ $\La_1$, \dots, $\La_n\in \K$.
Let $(\Phi_{l,\K}(s,x), \Psi_{m,\K}(s,x))$
be the rational functions in $s$ and $x$ with coefficients in $\K$ defined by formulas
\eqref{6a.1}. These functions satisfy the flatness conditions \eqref{6a.2}.

The corresponding scalar $p$-curvature functions are given by:
\begin{align*}
\tilde{\Phi}_{l,\K}(s,x) = \prod_{i=0}^{p-1} \Phi_{l,\K}(s + i ae_l, x),
\qquad
\tilde{\Psi}_{m,\K}(s,x) = \prod_{i=0}^{p-1} \Psi_{m,\K}(s, x + i a f_m).
\end{align*}

\begin{lem}
\label{lem 5.1}
The $p$-curvature functions $\left(\tilde{\Phi}_{l,\K}(s,x),\tilde{\Psi}_{m,\K}(s,x)\right)$
are $a$-periodic in variables $s$ and $x$.

\end{lem}

\begin{proof}
The lemma follows from the flatness conditions.
\end{proof}

\begin{lem}
\label{lem tilde}
The $p$-curvature functions $\left(\tilde{\Phi}_{l,\K}(s,x),\tilde{\Psi}_{m,\K}(s,x)\right)$ have the formulas:
\bean
\label{phil K}
\tilde\Phi_{l,\K}(s\<\>,x) &=&
\ka^p\,\prod_{i=1}^n\,\frac{h_p(s_l)-h_p(x_i)+h_p(\La_i)}{h_p(s_l)-h_p(x_i)}\,\cdot\<\>
\prod_{j\ne l}\,\frac{h_p(s_l)-h_p(s_j)-h_p(1)}{h_p(s_l)-h_p(s_j)+h_p(1)}\;,
\\[4pt]
\label{psim K}
\tilde\Psi_{m,\K}(s\<\>,x) &=& \prod_{l=1}^r\,\frac{h_p(s_l)-h_p(x_m)}{h_p(s_l)-h_p(x_m)+h_p(\La_m)}\;,
\eean
$l=1,\dots,r$, \;$m=1,\dots,n$\,, where $h_p(y) = y^p-a^{p-1}y$ is the
Artin–Schreier polynomial defined in Appendix \ref{app A}.

\end{lem}

\begin{proof}
The lemma follows from formulas of Appendix \ref{app A}.
\end{proof}

\subsection{Weight function over $\K$}

Let $\La_1$, \dots, $\La_n\in \K$. The $V_\K$-valued weight function,
\be
W_\K(s\<\>,x)\,=\>\sum_{\vec r\in \mc I_r} W_{\vec r}(s\<\>,x)\,f^{\<\>\vec r},
\ee
is defined by formula \eqref{wf2}. This is a $V_\K$-valued polynomial in $(s,x)$.

\subsection{Integral representation over $\K$}

We have the following corollaries of Theorems \ref{thm int rep} and \ref{thm sing}.

\begin{cor}
\label{cor int}
Let $a, \ka \in \K^\times$,\ $\La_1$, \dots, $\La_n\in \K$.
Then there exists $G_{l,m, \K}(s,x)$ for $l=1,\dots,r$, $m=1,\dots,n$, the
$V_\K$-valued rational functions in $(s,x)$, such that
\bean
\label{I qkz}
&&
\Psi_{m, \K} (s,x)\, W_\K(s,x + af_m)
\\
\notag
&&
\,=\, K_{m,\K}(x)W_\K(s,x) \,+\,
\sum_{l=1}^r \left(\Phi_{l,\K}(s,x)\,G_{l,m,\K}(s+ae_l, x) - G_{l,m,\K}(s,x)\right)
\eean
for $m=1,\dots,n$.
Furthermore, each of the coordinate functions of every $G_{l,m,\K}(s,x)$ in the basis $(f^{\vec r})$
is a ratio $\frac{P(s,x)}{Q(s,x)}$\,,
where $P(s,x)$ is a polynomial in $(s,x)$,
and $Q(s,x)$ is a product with factors each of which is of the form:
\begin{align*}
s_k-x_i+ca\,, &\quad k=1,\dots, r, \ i=1,\dots,n,
\\
s_k-x_i + \La_i +ca\,, &\quad k=1,\dots, r, \ i=1,\dots,n,
\\
s_k-s_j+ca \,, &\quad k,j = 1,\dots, r,\ k\ne j,
\\
s_k-s_j + 1+ca\,, &\quad k,j = 1,\dots, r,\ k\ne j,
\\
x_i-x_m-\La_i+b+ca, &\quad i,m=1,\dots,n, \ i\ne m,
\\
\ka,&
\end{align*}
where $c = 0, \dots, p-1$, \ $b= 0, \dots, r-1$.

\end{cor}

\begin{cor}
\label{cor sing}
Set $\ka=1$. Then there
exist $\left(\bigotimes^n_{j=1} M_{{\La_j, \K}}\right)_{r-1}$-valued
functions $H_{l,\K}(s,x)$ for $l=1,\dots,r$, rational in variables $s$ and $x$ such that
\bean
\label{SiK}
e\, W_\K(s,x) \,=\,
\sum_{l=1}^r \left(\Phi_{l,\K}(s,x)\,H_{l,\K}(s+ae_l, x) - H_{l,\K}(s,x)\right).
\eean
Furthermore, the coordinate functions of every $H_{l,\K}(s,x)$ in the basis
$\{f^{\vec u}\}_{\vec u\in\mathcal I_{r-1}}$,
$u_1+\cdots+u_n=r-1,$
of
$\left(\bigotimes^n_{j=1} M_{{\La_j},\K}\right)_{r-1}$
are ratios
$\frac{P(s,x)}{Q(s,x)}$\,,
where $P(s,x)$ is a polynomial in $(s,x)$,
and $Q(s,x)$ is the product of factors each of which has the form
$s_i-s_j$, where $i,j=1,\dots, r$, $i\ne j$.

\end{cor}

\section{Eigensections}
\label{sec 6}

In this section, we study the objects introduced in Section \ref{sec 5}.

\subsection{Orbits in $\A^n$}

Let $\A^n$ be the $n$-dimensional affine space over $\K$ with coordinates $x = (x_1,$ \dots, $x_n)$.
The group $\Z^n$ acts on $\A^n$ by translations:
\[
(k, x^0) \mapsto x^0 + a k,
\]
where $k \in \Z^n, x^0\in \A^n$, and $a k=(a k_1,\dots,a k_n).$
For any base point $x^0 \in \A^n$, we denote its orbit under this action by $\mc O_{x^0}$.

\vsk.2>

We call $x^0$ a regular point if the qKZ functions $(K_{m,\K})$
are regular and invertible on $\mc O_{x^0}$.

\vsk.2>

Let $x^0$ be regular.
Then we  pull back the qKZ functions $(K_{m,\K})$ to the integer lattice
by the action map $\Z^n \to \mc O_{x^0}$, $k \mapsto x^0 + a k$.
We denote the lifted functions by
$$
L_m(k) := K_{m,\K}(x^0 + a k).
$$
The lifted functions have properties \eqref{L per} and \eqref{L com}. Hence $(L_m)$ define
a $p$-invariant
discrete flat connection on the trivial bundle $\Z^n \times V_\K \to \Z^n$ in the sense of Section \ref{sec 2}.

\vsk.2>
We also lift the $p$-curvature functions $(\tilde K_{m,\K})$ and denote the lifted functions by $(C_m)$,
$$
C_m(k) := \tilde K_{m,\K}(x^0+ak).
$$
These functions are the $p$-curvature functions of the discrete connection $(L_m)$. They have properties
\eqref{c1} -- \eqref{c4}.

\vsk.2>

A function $I: \Z^n \to V_\K$ is called a
{\it multi-valued flat section} of the discrete connection $(K_{m,\K})$
over the orbit $\mc O_{x^0}$ if it satisfies:
\[
I(k + f_m) = {L}_m(k) I(k),
\]
for all $k\in \Z^n$ and $m=1, \dots, n$. It is also called a flat section of the discrete connection $(L_m)$.

\subsection{Orbits in $\A^r\times\A^n$}

Consider the affine space $\A^r$ with coordinates $s=(s_1,\dots,s_r)$. The group $\Z^{r+n}=\Z^{r}\times \Z^n$ acts on
$\A^r\times \A^n$
by translations:
\[
((j,k), (s^0,x^0)) \mapsto (s^0,x^0) + a(j, k),
\]
where $(j,k) \in\Z^r\times \Z^n$, $(s^0,x^0)\in\A^r\times \A^n$. For any base point $(s^0,x^0) \in \A^r\times\A^n$,
we denote its orbit under this action by $\mc O_{(s^0,x^0)}$.

\vsk.2>

Fix a choice of the rational functions $G_{l,m,\K}$ from Corollary~\ref{cor int}
 and, when $\ka=1$, of the functions $H_{l,\K}$ from Corollary~\ref{cor sing}.
We call $(s^0,x^0)$ a regular point (with respect to these choices) if $K_{m,\K}$, $\Phi_{l,\K}$, $\Psi_{m,\K}$, $G_{l,m,\K}$, and, in the case $\ka=1$, $H_{l,\K}$
are regular on $\mc O_{(s^0,x^0)}$ for all $l,m$, and $K_{m,\K}$, $\Phi_{l,\K}$, $\Psi_{m,\K}$
 are invertible on $\mc O_{(s^0,x^0)}$
for all $l, m$.

\vsk.2>
Given a regular point $(s^0,x^0)$, we pull back the functions $(\Phi_{l,\K}, \Psi_{m,\K})$ to the integer lattice
$\Z^{r}\times \Z^n$
by the action map $\Z^{r}\times \Z^n \to \mc O_{(s^0,x^0)}$, $(j,k) \mapsto (s^0,x^0) + a (j,k)$.
We denote the lifted functions by
\bea
\phi_l(j,k) := \Phi_{l,\K}((s^0,x^0) + a (j,k)), \qquad
\psi_m(j,k) := \Psi_{m,\K}((s^0,x^0) + a (j,k)).
\eea
The flatness conditions for these functions are given by formula \eqref{phi-phi}.
The functions $(\phi_l,\psi_m)$ are $p$-periodic.

\vsk.2>

We lift the $p$-curvature functions $(\tilde \Phi_{l,\K}, \tilde \Psi_{m,\K})$ and denote the lifted functions by
$(\tilde \phi_l, \tilde \psi_m)$. The functions $(\tilde \phi_l, \tilde \psi_m)$ are constant on $\Z^r\times \Z^n$, by Lemma
\ref{lem 5.1}.

\vsk.2>

We lift the weight function $W_\K$ to $\Z^r\times \Z^n$ and denote the lifted function by $w$.

\subsection{Bethe ansatz equations}

The Bethe ansatz equations are the system of equations
\bean
\label{bae 6}
\tilde \Phi_{l,\K}(s,x) = 1, \qquad l=1,\dots,r.
\eean
Let $(s^0, x^0)\in \A^r\times \A^n$ be a regular point which is a solution to \eqref{bae 6}. Then all points of
the orbit
$\mc O_{(s^0,x^0)}$ are also regular points, since regularity is a property of the orbit,
and they are solutions to \eqref{bae 6} by Lemma \ref{lem 5.1}.

\vsk.2>

If $(s^0, x^0)\in \A^r\times \A^n$ is a regular point which is a solution to \eqref{bae 6}, then
the functions $(\phi_l, \psi_m)$ lifted from $\mc O_{(s^0,x^0)}$ to $\Z^r\times \Z^n$
define a special $p$-invariant discrete flat connection in the
sense of Section \ref{sec spe}.

\subsection{Integral representation}

Let $(s^0,x^0)\in \A^r\times \A^n$ be a regular point which is  a solution to \eqref{bae 6}.
By lifting from the orbit $\mc O_{(s^0,x^0)}$\, to the lattice $\Z^r\times \Z^n$,
we obtained the functions $w$ and $(\phi_l, \psi_m)$.

\vsk.2>

Consider the second component $x^0$ of the solution $(s^0,x^0)$.
By lifting from the orbit $\mc O_{x^0}$ to the lattice $\Z^n$, we obtained the functions $(L_m)$.

\vsk.2>
Corollary \ref{cor int} implies the following statement.

\begin{lem}
\label{IInt}

There exist $p$-periodic functions
\[
g_{l,m}:\mathbb Z^r\times\mathbb Z^n\to V_\K,
\qquad l=1,\dots,r,\quad m=1,\dots,n,
\]
such that
\bean
\label{II qkz}
\phantom{aaaaaa}
\psi_m(j,k)\, w(j,k + f_m) \,=\, L_m(k)w(j,k) \,+\,
\sum_{l=1}^r \left(\phi_l(j,k)\,g_{l,m}(j+e_l, k) - g_{l,m}(j,k)\right)
\eean
for all $j\in\Z^r$, $k\in\Z^n$, $m=1,\dots,n$.

\end{lem}

\begin{proof}
The functions $(g_{l,m})$ are just the liftings of the functions $(G_{l,m,\K})$ in Corollary \ref{cor int}.
\end{proof}

Lemma \ref{IInt} provides an integral representation in the sense of Section \ref{sec 2.4}
for the $p$-invariant discrete flat connection $(L_m)$ associated with a regular solution $(s^0,x^0)$
to \eqref{bae 6}. The corresponding rank-one connection $(\phi_l,\psi_m)$ is special, since $(s^0,x^0)$
is a solution to \eqref{bae 6}.

\begin{thm}
\label{thm 5.6}

Let $(s^0,x^0)\in \A^r\times \A^n$ be a regular point which is  a solution to \eqref{bae 6}.
Let $(L_m)$ be associated with the solution $(s^0,x^0)$
to \eqref{bae 6}.
Define
a $V_\K$-valued function
$I : \Z^n \to V_\K,$
by the formula:
\bean
\label{ix3}
I(k)
&=&
\int w(j,k) dj\, .
\eean
Then the function $I$ is a flat section of the $p$-invariant discrete flat connection $(L_m)$,
\bean
\label{ili 3}
I(k+f_m) = L_m(k) I(k),
\eean
for all $k\in\Z^n$ and $m=1,\dots, n$. Furthermore, the function $I$, if nonzero,
is an eigensection of the $p$-curvature functions:
\bean
\label{e so 3}
C_m(k)I(k) = \tilde \psi_m(0,0)I(k),
\eean
for $k\in\Z^n$ and $m=1,\dots,n$.

If $\ka=1$, then the function $I:\Z^n \to V_\K$ takes values in $\Sing V_\K$, the singular vector subspace of $V_\K$.

\end{thm}

\begin{proof} Formulas \eqref{ili 3} and \eqref{e so 3}
are the result of the application of Theorems \ref{thm sol} and \ref{thm i sol}
to $(L_m)$.
The last statement of the theorem means that
\bea
e\,I(k) = 0, \qquad k\in\Z^n.
\eea
That follows from Corollary \ref{cor sing} and Lemma \ref{lem int d}.
\end{proof}

We reformulate Corollary \ref{cor eigen} in terms of the original qKZ functions $(K_{m,\K})$.

\begin{cor}
\label{cor m eigen}

Let $(s^0, x^0)$ be a regular point which is a solution to the system of Bethe ansatz equations:
\bean
\label{BEA}
\ka^p\,\prod_{i=1}^n\,\frac{h_p(s_l)-h_p(x_i)+h_p(\La_i)}{h_p(s_l)-h_p(x_i)}\,\cdot\<\>
\prod_{j\ne l}\,\frac{h_p(s_l)-h_p(s_j)-h_p(1)}{h_p(s_l)-h_p(s_j)+h_p(1)}\,=\,1,
\eean
for $l=1,\dots, r$. Define a vector in $V_\K$,
\bean
\label{mx0}
J(s^0,x^0)
&=&
\sum_{i_1, \dots, i_r=0}^{p-1}
W_\K(s^0 + a(i_1 e_1+\dots + i_r e_r), x^0)
\\
\notag
&&
\phantom{aa}
\times \ \
\prod_{l=1}^r\prod_{d_l=0}^{i_l-1}
\Phi_{l,\K}(s^0 + a(i_1e_1+\dots +i_{l-1}e_{l-1} +d_le_l), x^0).
\eean
Then $J(s^0,x^0)$, if nonzero, is an eigenvector of the $p$-curvature operators $\tilde K_{m,\K}(x^0)$,
$m=1,\dots,n$, with respective eigenvalues
\beq
\label{Mon psi}
\prod_{l=1}^r\,\frac{h_p(s_l^0)-h_p(x_m^0)}{h_p(s_l^0)-h_p(x_m^0)+h_p(\La_m)}\,.
\eeq

Parallel transport of the fiber $V_\K$ over $x^0$ to the fibers over the points of the orbit $\mc O_{x^0}$
extends the vector $J(s^0,x^0)$ to a multi-valued section over $\mc O_{x^0}$ which is an eigensection of the
$p$-curvature functions $(\tilde K_{m,\K})$.

\end{cor}

\section{Special cases}
\label{sec 7}

To illustrate our results, we consider special cases of the discrete qKZ connection \eqref{Kmx} on
\bea
\tsize V_\K=\left(\bigotimes^n_{j=1} M_{{\La_j,\K}}\right)_r\,.
\eea
We assume that
$\La_{1},\, \dots, \,\La_{n} \, \in \F_p\subset \K$\,,
and analyze different outcomes depending on $a$ and $\ka$.

\subsection{Polynomial flat sections}

Assume that
\bean
\label{7.2}
\La_{1},\, \dots, \,\La_{n} \, \in \F_p\subset \K\,, \qquad a\in \F_p^\times , \ \qquad \ka=1.
\eean

\subsubsection{Master function}

For \,$b\in\F_p$\,, \,let \,$\bar b\in\{1,\ldots, p\>\}$
\,be such that \,$a\bar b=b\pmod p$\,. Set
\be
g(x,b)\,=\,h_{\>\bar b-1}(x+b-a)\,:=\,(x+b-a)\dots(x+2\<\>a)\>(x+a)\,.
\ee
For $\bar b=1$, this is an empty product, so $g(x,b)=1$.
Then
\be
\frac{g(x+a,b)}{g(x,b)}\,=\,\frac{x+b}{x+a}\;.
\ee

\begin{lem}
Under assumptions \eqref{7.2}, the function
\be
F(s\<\>,x)\,=\,\prod_{m=1}^n\,\prod_{l=1}^r\,g(s_l-x_m,\La_{m})\,\cdot\<\>
\prod_{1\le l<j\le r}\,\bigl((s_l-s_j)\>g(s_l-s_j+1,-\<\>2)\bigr)\,,
\ee
is a polynomial in $s$ and $x$. Furthermore,
the function $F(s\<\>,x)$ is a master function for the connection
$(\Phi_{l,\K}, \Psi_{m,\K})$ defined in Section \ref{sec 5.5},
that is,
\be
\Phi_{l,\K}(s,x)\,=\,\frac{F(s + a e_l,x)}{F(s,x)}\;,
\qquad
\Psi_{m,\K}(s,x)\,=\,\frac{F(s, x+ af_m)}{F(s,x)}\;,
\ee
for $l=1,\dots,r$ and $m=1,\dots, n$.
\end{lem}

\begin{proof}
The proof is by inspection.
\end{proof}

\subsubsection{Bethe ansatz equations}

Under assumptions \eqref{7.2},
we have $h_p(\La_{l})=h_p(1)=0$ for $l=1,\dots,n$, and
\bea
\tilde\Phi_{l,\K}(s,x)
&=&
\ka^p\,\prod_{i=1}^n\,\frac{h_p(s_l)-h_p(x_i)+h_p(\La_i)}{h_p(s_l)-h_p(x_i)}\,\cdot\<\>
\prod_{j\ne l}\,\frac{h_p(s_l)-h_p(s_j)-h_p(1)}{h_p(s_l)-h_p(s_j)+h_p(1)}\,=\,1,
\eea
that is, the Bethe ansatz equations \eqref{bae 6} hold identically as rational equations,
so every regular point $(s^0,x^0)$ is a solution.
Furthermore, we have:
\bea
\tilde\Psi_{m,\K}(s,x)
&=&
\prod_{l=1}^r\,\frac{h_p(s_l)-h_p(x_m)}{h_p(s_l)-h_p(x_m)+h_p(\La_m)}\,=1, \qquad m=1,\dots,n\,,
\eea
that is, the functions $(\tilde \Psi_{m,\K})$ are constants equal to 1.\footnote{${}$ We assumed that $\ka=1$. Hence
$\ka^p=1$.
Conversely, if $\ka^p=1$ in $\K$, then $\ka=1$ since
$\ka^p-1=(\ka-1)^p$.}

\subsubsection{Polynomial flat sections}

The function $W_\K(s,x)\,F(s,x)\,$ is a $V_\K$-valued polynomial.
Consider the expansion
\bea
W_\K(s,x)\>F(s,x)\,=\!\sum_{i_1,\dots,\>i_r\in\Z_{\geq 0}}\!
Q_{i_1,\dots,\>i_r}(x)\;h_{i_1}(s_1)\!\<\dots h_{i_r}(s_r)\>.
\eea

\begin{thm}
\label{thm Lau}
For any $l_1,\dots,l_r\in \Z_{\geq 0}$, the $V_\K$-valued polynomial
$Q_{pl_1+p-1,\dots,\>pl_r+p-1}(x)$ is a flat section
of the qKZ connection \eqref{ff qKZ},
\bean
\label{flat lp}
Q_{pl_1+p-1,\dots,\>pl_r+p-1}( x+af_m) = K_{m,\K}(x) Q_{pl_1+p-1,\dots,\>pl_r+p-1}(x), \quad m=1,\dots,n.
\eean
Furthermore, $Q_{pl_1+p-1,\dots,\>pl_r+p-1}( x)$ is an eigensection of the $p$-curvature operators with
all eigenvalues equal to 1, that is,
\bean
\label{7.4}
\tilde K_{m,\K}(x) \,Q_{pl_1+p-1,\dots,\>pl_r+p-1}(x) = Q_{pl_1+p-1,\dots,\>pl_r+p-1}(x),
\quad m=1,\dots,n.
\eean
Also,
\bean
\label{7.5}
e\, Q_{pl_1+p-1,\dots,\>pl_r+p-1}= 0.
\eean

\end{thm}

\begin{proof}

Define
\be
I(s,x)\,=\!\sum_{j_1, \dots, j_r=0}^{p-1}
W_\K(s+a(j_1e_1+\dots +j_r e_r),x)\,F(s+a(j_1e_1+\dots +j_r e_r),x)\,.
\ee

On the one hand, it follows from the identities in Appendix \ref{app A},
that
\bea
I(s,x)\,= (-1)^r\! \sum_{l_1,\dots,l_r\in\Z_{\geq 0}}
Q_{pl_1+p-1,\dots,\>pl_r+p-1}(x) \,\prod_{i=1}^r h_p(s_i)^{l_i}\,.
\eea

On the other hand, by Theorem \ref{thm 5.6},
the $V_\K$-valued polynomial
$I(s,x)$ is a flat section of the qKZ connection,
\bea
I(s,x+af_m)\,=\, K_{m,\K}(x) \,I(s,x), \qquad m=1,\dots,n.
\eea

These two observations imply property \eqref{flat lp}. Property
\eqref{7.4} follows from Theorem \ref{thm i sol}.
Property \eqref{7.5} holds since $\ka=1$.
\end{proof}

\begin{rem}

Theorem \ref{thm Lau} was proved in \cite{MV1} for $\La_i=1$, $i=1,\dots,n$, using a different construction.
The polynomial
flat sections $(Q_{pl_1+p-1,\dots,\>pl_r+p-1})$ were called in \cite{MV1} $p$-hypergeometric. See also \cite{MV2}.

\end{rem}

Recall that
\bea
\La_{1},\, \dots, \,\La_{n} \, \in \F_p\subset \K\,, \quad a\in \F_p^\times , \ \quad \ka=1.
\eea

\begin{exmp}

Assume that \,$\La_{i}=1$ for $i=1,\dots,n$\,, and \,$r=1$\,. Let $a=-2$ and $n=2g+1$ for some positive integer $g$. Then Theorem
\ref{thm Lau} provides $g$ $p$-hypergeometric flat sections. They are linearly independent by \cite{MV2}.

\end{exmp}

\subsection{No solutions}

Assume that
\bean
\label{7.7}
\La_{1},\, \dots, \,\La_{n} \, \in \F_p\subset \K\,, \qquad a\in \F_p^\times , \qquad \ka\ne 1.
\eean
Then $\ka^p\ne 1$.
\vsk.2>

Under these assumptions,
we have $h_p(\La_{l})=h_p(1)=0$, and
$\tilde\Phi_{l,\K}(s,x)
=
\ka^p$\,.
That means that the system of Bethe ansatz equations \eqref{bae 6} has no solutions and the method of
Sections \ref{sec 2} and \ref{sec 3} does not apply.

\subsection{Solutions as orbits and the XXX model}

Assume that
\bean
\label{7.8}
\La_{1},\, \dots, \,\La_{n} \, \in \F_p\subset \K\,, \qquad a\in \K\setminus \F_p.
\eean
Then $h_p(1)\ne 0$. Introduce new variables
\be
t_l=\frac{h_p(s_l)}{h_p(1)}\,,\quad z_m=\frac{h_p(x_m)}{h_p(1)}\,,\quad
\la_m=\frac{h_p(\La_{m})}{h_p(1)}\,,\quad \al=\ka^p\>.
\ee
Notice that $\la_m=\La_m \in \F_p$.
Then
\bean
\label{tTh}
\tilde\Phi_{l,\K}(s\<\>,x) &=&
\al\,\prod_{i=1}^n\,\frac{t_l-z_i+\la_i}{t_l-z_i}\,\cdot\<\>
\prod_{j\ne l}\,\frac{t_l-t_j-1}{t_l-t_j+1}\;,
\\[4pt]
\label{tTs}
\tilde\Psi_{m,\K}(s\<\>,x) &=& \prod_{l=1}^r\,\frac{t_l-z_m}{t_l-z_m+\la_m}\;.
\eean

Solutions to the system of Bethe ansatz equations \eqref{bae 6},
\bean
\label{ph ba}
\ka^p\,\prod_{i=1}^n\,\frac{h_p(s_l)-h_p(x_i)+h_p(\La_i)}{h_p(s_l)-h_p(x_i)}\,\cdot\<\>
\prod_{j\ne l}\,\frac{h_p(s_l)-h_p(s_j)-h_p(1)}{h_p(s_l)-h_p(s_j)+h_p(1)}\,=\,1\;,
\eean
$l=1,\dots,r$, can be obtained in two steps: first, we solve the system:
\bean\label{tba}
\al\,\prod_{i=1}^n\,\frac{t_l-z_i+\la_i}{t_l-z_i}\,\cdot\<\>
\prod_{j\ne l}\,\frac{t_l-t_j-1}{t_l-t_j+1}\; = 1,\qquad l=1,\dots, r.
\eean
Then for any solution $(t^0,z^0)=(t^0_1,\dots, t^0_r, z_1^0,\dots,z_n^0)$ to system \eqref{tba}, we solve the system

\bean
\label{sba}
\frac{h_p(s_l)}{h_p(1)} \,=\, t_l^0,\quad l=1,\dots,r, \qquad
\frac{h_p(x_m)}{h_p(1)} \,=\, z_m^0,\quad m=1,\dots,n.
\eean

\smallskip

\begin{lem}
\label{lem orbits}
If $(s^0, x^0)$ is any solution to system \eqref{sba}, then the set of all solutions to system
\eqref{sba} consists of the $p^{r+n}$ points\,:
\bean
\label{orb s}
(s^0+a(i_1e_1+\dots+i_re_r), \ x^0+a(j_1f_1+\dots+j_n f_n)),
\eean
where $0\leq i_1,\dots,i_r, j_1,\dots,j_n \leq p-1.$\
In other words, the set of all solutions to system \eqref{sba} is the orbit $\mc O_{(s^0, x^0)}$
of the $\Z^{r}\times \Z^n$-action on $\A^{r}\times \A^n$.

\end{lem}

Therefore, the system of Bethe ansatz equations \eqref{bae 6} can be considered as a
system of equations on the space of orbits of the $\Z^{r}\times \Z^n$-action on $\A^{r}\times \A^n$.

\begin{rem}

System \eqref{tba} arises in the XXX quantum integrable chain
as the system of Bethe ansatz equations when
one diagonalizes the commuting transfer matrices of the chain. 
To every solution $(t^0,z^0)$ of that system, one assigns
an eigenvector of the transfer matrices with eigenvalues given by the formula
$\left(\prod_{l=1}^r\,\frac{t^0_l-z^0_m}{t^0_l-z^0_m+\la_m}\right)\,.$

\end{rem}

\appendix

\section{Pochhammer polynomials}
\label{app A}

Let $\K$ be a field of characteristic $p>2$. Let $a\in \K^\times$. Let $y$ be a variable.
For a nonnegative integer $k$, define the Pochhammer polynomial with step $a$ by the formula:
\bea
h_k(y)=y\>(y-a)\dots(y-(k-1)\>a)\,, \quad \,k\geq 1\,,\qquad h_0(y)=1.
\eea
The polynomial
\bea
h_p(y)=y\>(y-a)\dots(y-(p-1)\>a) = y^p-a^{p-1}y
\eea
is called the {\it Artin–Schreier polynomial}.

We have the following identities:
\bea
h_{k+mp}(y)
&=&h_k(y)\>h_p(y)^m \,,
\\
h_p(y+z)&=&h_p(y)+h_p(z)\,,
\\
h_p(-\<\>y) &=& -\<\>h_p(y) \,,
\\
\>(k+1)\<\>a\,h_{k}(y)\>
&=&
h_{k+1}(y+a)-h_{k+1}(y)\,,
\\
(k+1)\>a\dsize\sum_{i=0}^{j-1}\,h_k(y+i\<\>a)\,
&=&
h_{k+1}(y+ja)-h_{k+1}(y),
\\
\sum_{i=0}^{p-1}\,h_k(y+i\<\>a)\,
&=&
\,0\ \ \on{if}\ \ p\nmid(k+1)\,,
\\
\sum_{i=0}^{p-1}\,h_{mp-1}(y+i\<\>a)\,
&=&
\,-\>a^{p-1} h_p(y)^{m-1} \, \ \on{if}
\ m\in\Z_{\ge1}\,.
\eea

\section{Eigenvectors of $p$-curvature operators form a basis}
\label{appB}

In Corollary~\ref{cor m eigen}, we considered the degree-$r$ subspace
\[
V_\K
=
\left(\bigotimes_{j=1}^n M_{\La_j,\K}\right)_r
\]
and the qKZ discrete connection $(K_{m,\K})$.
For any solution $(s^0,x^0)$ of the Bethe ansatz equations~\eqref{BEA},
we constructed a vector $J(s^0,x^0)\in V_\K$ and showed that, whenever
this vector is nonzero, it is a common eigenvector of the $p$-curvature
operators $(\tilde K_{m,\K})$.

In this appendix, we restrict to the case $r=1$, so that
\[
V_\K
=
\left(\bigotimes_{j=1}^n M_{\La_j,\K}\right)_1,
\qquad
\dim V_\K=n.
\]
Under suitable genericity assumptions, we prove that our construction gives a basis of $V_\K$ consisting of
common eigenvectors of the $p$-curvature operators.

\subsection{Rank-one functions}

For $r=1$, $s$ is a single variable and $x=(x_1,\dots,x_n)$.
The rank-one functions are
\bean
\label{BPhi}
\Phi_{1,\K}(s\<\>,x) &=&
\ka\;\prod_{i=1}^n\,\frac{s-x_i+\La_i}{s-x_i+a}\,,
\\
\label{BPsi}
\Psi_{m,\K}(s\<\>,x) &=& \frac{s-x_m}{s-x_m+\La_m-a}\;,
\eean
\,$m=1,\dots,n$\,. These are rational functions in $a, \ka, \La, s, x$.
The corresponding $p$-curvature functions are
\bean
\label{Bphil K}
\tilde\Phi_{1,\K}(s\<\>,x) &=&
\ka^p\,\prod_{i=1}^n\,\frac{h_p(s)-h_p(x_i)+h_p(\La_i)}{h_p(s)-h_p(x_i)}
%\\[4pt] \notag
=
\ka^p\,\prod_{i=1}^n\,\frac{h_p(s-x_i+\La_i)}{h_p(s-x_i)}\;,
\\
\label{Bsim K}
\tilde\Psi_{m,\K}(s\<\>,x) &=&
\frac{h_p(s)-h_p(x_m)}{h_p(s)-h_p(x_m)+h_p(\La_m)} =
\frac{h_p(s-x_m)}{h_p(s-x_m+\La_m)}\;,
\eean
where \;$m=1,\dots,n$\, and $h_p(u) = \prod_{j=0}^{p-1}(u - aj) = u^p-a^{p-1}u$.

The Bethe ansatz equation is $\tilde\Phi_{1,\K}(s,x)=1$, that is,
\bean
\label{BBAE}
\ka^p\,\prod_{i=1}^n\,\frac{h_p(s-x_i+\La_i)}{h_p(s-x_i)}\,=\, 1.
\eean

\subsection{Weight function}

Let
\[
f^{(j)}
=
v_{\La_1}\otimes\cdots\otimes
fv_{\La_j}\otimes\cdots\otimes v_{\La_n},
\qquad j=1,\dots,n.
\]
The weight function is
\begin{equation}
\label{BW1}
W_\K(s,x)=\sum_{j=1}^n W_{j,\K}(s,x)\,f^{(j)},
\end{equation}
where
\begin{equation}
\label{BW2}
W_{j,\K}(s,x)
=
\prod_{l<j}(s-x_{l}+\La_{l})
\prod_{l>j}(s-x_{l}).
\end{equation}
The Bethe vector~\eqref{mx0} takes the form
\begin{equation}
\label{mx1}
J(s,x)
=
\sum_{d=0}^{p-1}
W_\K(s+ da, x)\,
\prod_{j=0}^{d-1}\Phi_{1,\K}(s+ aj, x)\,.
\end{equation}
Note that $J(s,x)$ and $\tilde K_{m,\K}(x^0)$ also depend on $\La$
and $\ka$, but we suppress this dependence.

Corollary~\ref{cor m eigen} for $r=1$ takes the following form.

\begin{cor}
\label{cor B}

Let $(s^0,x^0)$ be a regular point which is a solution of the Bethe ansatz equation \eqref{BBAE}.
Then $J(s^0,x^0)$,
whenever nonzero, is a common eigenvector of the $p$-curvature operators
$(\tilde K_{m,\K}(x^0))$, $m=1,\dots,n$, with eigenvalues
$(\tilde\Psi_{m,\K}(s^0,x^0))$ given by~\eqref{Bsim K}.

\end{cor}

\subsection{Asymptotics as $\ka\to\infty$}

We fix the parameters $\La=(\La_1,\dots,\La_n)$, $a$, and $x^0=(x_1^0,\dots,x_n^0)$ in two steps.
First, we fix
\bean
\label{a1}
\La=(\La_1,\dots,\La_n)\in (\F_p^\times)^n\subset\K^n, \qquad \ a \in \K\setminus\F_p\,.
\eean
This implies that for every $i=1,\dots, n$ all $2p$ linear factors in $s$ in the numerator and denominator of
\bea
\prod_{j=0}^{p-1}\frac{s-x_i + \La_i + aj}{s-x_i + aj}
\eea
are pairwise distinct. Next, choose $x^0=(x_1^0,\dots,x_n^0)\in \K^n$ such that
\bean
\label{a2}
&&
\text{All}\ 2np\ \text{linear factors in}\,s\, \text{in the numerator and denominator of}
\\
\notag
&&
\phantom{aaa}
\prod_{i=1}^n\prod_{j=0}^{p-1}\frac{s-x_i^0 + \La_i + aj}{s-x_i^0 + aj} \ \
\text{are \ pairwise\ distinct.}
\eean
Set $\eps=\ka^{-1}$ and let $\eps\to 0$.

\smallskip

Rewrite~\eqref{BBAE} as
\begin{equation}
\label{BAE1}
\prod_{i=1}^n\prod_{j=0}^{p-1}(s-x_i^0 + \La_i + aj)
\,=\,\eps^p
\prod_{i=1}^n\prod_{j=0}^{p-1}(s-x_i^0 + aj).
\end{equation}
For $\eps=0$, this is a polynomial equation of degree $pn$ in $s$
with distinct simple roots. We are interested in $n$ of the roots:
\[
y_i^0:=x_i^0-\La_i\,,\qquad i=1,\dots,n.
\]
Consider the affine plane with coordinates $(\eps, s)$ and the affine curve $X$ defined by
equation \eqref{BAE1}.

Define
\[
F(\eps,s)
=
\prod_{i=1}^n\prod_{j=0}^{p-1}
(s-x_i^0+\La_i+aj)
-
\eps^p
\prod_{i=1}^n\prod_{j=0}^{p-1}
(s-x_i^0+aj).
\]
Then
$F(0,y_i^0)=0.$
By \eqref{a2}, \(y_i^0=x_i^0-\La_i\) is a simple root of the first product, and hence
\[
\frac{\partial F}{\partial s}(0,y_i^0)\neq0.
\]
Therefore, the formal implicit-function theorem gives a unique formal branch
\[
s_i(\eps)\in\K[[\eps]],
\qquad
s_i(0)=y_i^0.
\]
To prove the stronger estimate, write$$s_i(\eps)=y_i^0+u_i(\eps).$$Since the first product has a simple zero at $y_i^0$, it is$$c_i u_i(\eps)+\mc O(u_i(\eps)^2),
\qquad c_i\neq 0.$$The second term in $F$ is divisible by $\eps^p$. Consequently,$$u_i(\eps)=\mc O(\eps^p),$$and hence$$s_i(\eps)=y_i^0+\mc O(\eps^p), \qquad i=1,\dots,n.$$
Because the constants $y_i^0$ are pairwise distinct, these branches are distinct for generic nonzero $\eps$.

\begin{lem}
\label{lem:reg}
Assume \eqref{a1} and \eqref{a2}. Then $x^0$ is a regular point. Moreover, if $\eps\ne0$ and $s^0$ is a solution of
\eqref{BAE1}, then $(s^0,x^0)$ is a regular point.
\end{lem}

\begin{proof}
For $r=1$, the matrix entries of $R_{\La_m,\La_i}(x)$ and of $\bigl(R_{\La_m,\La_i}(x)\bigr)^{-1}$ on the degree-one subspace
have the denominators $x-\La_m$ and $x+\La_i$, respectively. Since $\ka\ne0$, the qKZ functions
$K_{m,\K}$ are regular and invertible on $\mc O_{x^0}$ if
$x_m^0-x_i^0-\La_m\notin a\F_p$ and $x_m^0-x_i^0+\La_i\notin a\F_p$ for all $i\ne m$.
The first condition means that the factors $s-x_m^0+\La_m+aj$ and $s-x_i^0+aj'$ in \eqref{a2} are distinct,
and the second condition means that the factors $s-x_i^0+\La_i+aj$ and $s-x_m^0+aj'$ are distinct.
Hence $x^0$ is regular.

In the case $\ka=1$, the function $H_{1,\K}$ of Corollary \ref{cor sing} is a polynomial in $(s,x)$,
since for $r=1$ its
 denominators contain no factors $s_i-s_j$. By formulas \eqref{BPhi}, \eqref{BPsi} and
Corollary \ref{cor int}, the functions $\Phi_{1,\K}$, $\Psi_{m,\K}$, $G_{1,m,\K}$ are regular, and the functions
$\Phi_{1,\K}$, $\Psi_{m,\K}$ are invertible, on $\mc O_{(s^0,x^0)}$
if, in addition, $s^0$ is not a root of any of the $2np$ linear factors listed in \eqref{a2}.
If $\eps\ne0$ and $s^0$ is a root of a factor in the left-hand side of \eqref{BAE1}, then
$s^0$ is also a root of a factor in the right-hand side, and vice versa. This contradicts \eqref{a2}.
\end{proof}

\begin{lem}
\label{lem:B1}
For $i=1,\dots,n$ and $d=1,\dots,p-1$,
along the formal branch $s=s_i(\eps)$, we have
$$\prod_{j=0}^{d-1}
\Phi_{1,\K}(s_i(\eps)+aj,x^0)
=\mc O(\eps^{p-d}).$$
In particular, this product is $\mc O(\eps)$.
\end{lem}

\begin{proof}

By \eqref{a2}, all denominators occurring in the factors $\Phi_{1,\K}(s_i(\eps)+aj,x^0)$ are nonzero at $(0,y_i^0)$.
The factor with $j=0$ is $\mc O(\eps^{p-1})$, while each factor with $1\le j\le d-1$ is $\mc O(\eps^{-1})$.

Multiplying all factors gives
\[
\prod_{j=0}^{d-1}\Phi_{1,\K}(s_i(\eps)+aj, x^0)
=
\mc O(\eps^{p-1})\cdot\mc O(\eps^{-1})^{d-1}
=
\mc O(\eps^{p-d}).
\]
\end{proof}

\begin{lem}
\label{lem:B2}

Since $W_\K(s,x^0)$ is polynomial in $s$, the function $W_\K(s_i(\eps)+ad,x^0)$ is regular along the branch $s=s_i(\eps)$
for $i=1,\dots, n$ and $d=0,\dots,p-1$. The coordinates of $W_\K(y_i^0,x^0)$ have the properties:
\begin{align}
W_{j,\K}(y_i^0, x^0)
&=0,
\qquad
j>i,
\label{B15}
\\
W_{i,\K}(y_i^0, x^0)
&=
\prod_{l<i}\!\left( x_i^0-\La_i -x_l^0+\La_l \right)
\prod_{l>i}\!\left( x_i^0-\La_i -x_l^0\right)
\ne0.
\label{B16}
\end{align}
\end{lem}

\begin{proof}
For $j>i$, the first product defining $W_{j,\K}$ 
contains the factor with $l=i$:$$s-x_i^0+\La_i.$$At $s=y_i^0=x_i^0-\La_i$, this factor is zero.
For $j=i$, neither product contains this vanishing factor, and the nonzero value is exactly \eqref{B16}.

\end{proof}

\begin{cor}
\label{cor:lim}

For each $i=1,\dots,n$, the pullback $J(s_i(\eps),x^0)$
is regular at $\eps=0$, and
$$
J(s_i(\eps),x^0)
=
W_\K(y_i^0,x^0)+\mc O(\eps).
$$
The vectors $W_\K(y_i^0,x^0)$, $i=1,\dots,n$, form a basis of $V_\K$.
\end{cor}

\begin{proof}
Along the branch $s=s_i(\eps)$, the $d=0$ term in \eqref{mx1} is
$W_\K(s_i(\eps),x^0),$ which evaluates to $W_\K(y_i^0,x^0)$ at $\eps=0$.
For every $d\ge 1$, Lemma \ref{lem:B1} gives$$\prod_{j=0}^{d-1}
\Phi_{1,\K}(s_i(\eps)+aj,x^0)=\mc O(\eps).$$
Since $W_\K(s_i(\eps)+ad,x^0)$ is regular, each term with $d\ge1$ tends to zero. Therefore$$J(s_i(\eps),x^0)
=
W_\K(y_i^0,x^0)+\mc O(\eps).$$

The vectors $W_\K(y_i^0,x^0)$ form a basis because the matrix of their coordinates in the basis$$f^{(1)},\dots,f^{(n)}$$is triangular with nonzero diagonal entries, by \eqref{B15} and \eqref{B16}.
\end{proof}

The previous lemmas imply the following theorem.

\begin{thm}
\label{thm:basis}
Assume that $\K$ is algebraically closed.
Assume \eqref{a1} and \eqref{a2}.
Then for generic $\ka\in\K^\times$, equation \eqref{BBAE} has $n$ distinct solutions $s_1^0,\dots, s_n^0$, obtained by deforming the roots $y_1^0,\dots,y_n^0$, such that the corresponding vectors
$J(s^0_i,x^0)$, $i=1,\dots,n$, are nonzero and form a basis of $V_\K$.
Moreover, each $J(s_i^0,x^0)$ is a common eigenvector of the $p$-curvature
operators $(\tilde K_{m,\K}(x^0))$ with respective eigenvalues
\[
\frac{h_p(s_i^0-x_m^0)}{h_p(s_i^0-x_m^0+\La_m)}\,,
\qquad m=1,\dots,n,
\]
that is,
\bea
\tilde K_{m,\K}(x^0)J(s_i^0,x^0)
\,=\,
\,\frac{h_p(s_i^0-x_m^0)}
{h_p(s_i^0-x_m^0+\La_m)}
J(s_i^0,x^0).
\eea

\end{thm}

\begin{proof}
For each $i=1,\dots,n$, let $s_i(\eps)$ denote the formal branch of the
curve $X$ through $(0,y_i^0)$ constructed above. Thus
\[
F\bigl(\eps,s_i(\eps)\bigr)=0,
\qquad
s_i(0)=y_i^0,
\qquad
s_i(\eps)=y_i^0+\mc O(\eps^p).
\]
Since the points $y_1^0,\dots,y_n^0$ are pairwise distinct, these branches
are pairwise distinct.

Consider the fiber product
\[
Y=X\times_{\mathbb A^1_{\eps}}\cdots
\times_{\mathbb A^1_{\eps}}X,
\]
with coordinates $(\eps,s_1,\dots,s_n)$, where each coordinate satisfies
$F(\eps,s_i)=0$. Let
\[
Q_0=(0,y_1^0,\dots,y_n^0)\in Y.
\]
By the simplicity of the roots $y_i^0$, the projection
\[
\pi:Y\longrightarrow \mathbb A^1_{\eps},
\qquad
(\eps,s_1,\dots,s_n)\longmapsto\eps,
\]
is \'etale in a neighborhood of $Q_0$. In particular, after replacing $Y$
by a sufficiently small open neighborhood of $Q_0$, the coordinates $s_i$
represent the $n$ branches chosen above.

By Corollary~\ref{cor:lim}, for each $i$ the pullback of
$J(s_i,x^0)$ to this neighborhood has a regular extension to $\eps=0$,
and its value at $Q_0$ is
\[
J(s_i,x^0)\big|_{Q_0}=W_\K(y_i^0,x^0).
\]
After shrinking to a common Zariski-open neighborhood $U\subset Y$ of $Q_0$,
we may assume that all these extensions are regular on $U$ and that
$\pi$ is \'etale on $U$.
Define the determinant
\[
D
=
\det\bigl(J(s_1,x^0),\dots,J(s_n,x^0)\bigr),
\]
where the determinant is taken with respect to the basis
$f^{(1)},\dots,f^{(n)}$ of $V_\K$. Then $D$ is regular on $U$, and
\[
D(Q_0)
=
\det\bigl(W_\K(y_1^0,x^0),\dots,W_\K(y_n^0,x^0)\bigr)\neq0,
\]
because the vectors $W_\K(y_i^0,x^0)$ form a basis of $V_\K$ by
Corollary~\ref{cor:lim}.

Consequently, the nonvanishing locus
\[
Y_D=\{Q\in U:D(Q)\neq0\}
\]
is a nonempty Zariski-open neighborhood of $Q_0$ in $U$ (and hence in $Y$).
Since $\pi$ is \'etale on $U$, it is an open map on this neighborhood. Hence $\pi(Y_D)$ contains a
nonempty Zariski-open subset of the $\eps$-line. For every $\eps$ in this
open subset with $\eps\neq0$, there is a point
\[
Q=(\eps,s_1,\dots,s_n)\in Y_D.
\]
The numbers $s_1,\dots,s_n$ are distinct solutions of the Bethe ansatz
equation \eqref{BBAE}, obtained by deforming the roots
$y_1^0,\dots,y_n^0$, and
\[
D(Q)\neq0.
\]
Therefore the vectors
\[
J(s_1,x^0),\dots,J(s_n,x^0)
\]
are nonzero and form a basis of $V_\K$.

Since $\eps=\ka^{-1}$, a nonempty Zariski-open subset of the
$\eps$-line away from $\eps=0$ corresponds to a nonempty Zariski-open
subset of $\K^\times$ in the parameter $\ka$. Thus the asserted basis
property holds for generic $\ka\in\K^\times$.

Finally, since $\eps\ne0$, each point $(s_i,x^0)$ is regular by Lemma~\ref{lem:reg}.
Hence, by Corollary~\ref{cor B}, each vector $J(s_i,x^0)$ is a common
eigenvector of the $p$-curvature operators
$\tilde K_{m,\K}(x^0)$, with eigenvalue
\[
\frac{h_p(s_i-x_m^0)}
{h_p(s_i-x_m^0+\La_m)}
\]
for $m=1,\dots,n$. This proves the theorem.
\end{proof}

\end{document}